\documentclass[10pt]{article}

\usepackage{amsmath, graphicx, latexsym, amssymb, amsthm, amsfonts}
\usepackage[mathscr]{eucal}
\usepackage{graphics}
\usepackage{color}
\usepackage{epsfig}
\usepackage{hyperref}
\usepackage{enumerate}
\usepackage{setspace}

\newtheorem{lemma}{Lemma}[section]
\newtheorem{theorem}{Theorem}[section]

\newtheorem{remark}{Remark}[section]
\newtheorem{example}{Example}[section]

\def\thelemma{\arabic{section}.\arabic{lemma}}
\def\thetheorem{\arabic{section}.\arabic{theorem}}
\def\thecorollary{\arabic{section}.\arabic{corollary}}
\def\thedefinition{\arabic{section}.\arabic{definition}}
\def\theexample{\arabic{section}.\arabic{example}}
\def\theproposition{\arabic{section}.\arabic{proposition}}

\def\theassumption{\arabic{section}.\arabic{assumption}}

\def\theremark{\arabic{section}.\arabic{remark}}

\newcommand{\manualnames}[1]{
\def\thelemma{#1.\arabic{lemma}}
\def\thetheorem{#1.\arabic{theorem}}
\def\thecorollary{#1.\arabic{corollary}}
\def\thedefinition{#1.\arabic{definition}}
\def\theexample{#1.\arabic{example}}
\def\theproposition{#1.\arabic{proposition}}
\def\theassumption{#1.\arabic{assumption}}
\def\theremark{#1.\arabic{remark}}
}

\newcommand{\beginsec}{
}

\numberwithin{equation}{section}

\newcommand{\la}{\lambda}

\newcommand{\al}{\alpha}

\newcommand{\gam}{\gamma}

\newcommand{\s}{\sigma}
\newcommand{\sig}{\sigma}
\newcommand{\del}{\delta}

\newcommand{\Gam}{\mathnormal{\Gamma}}
\newcommand{\Del}{\mathnormal{\Delta}}

\newcommand{\La}{\mathnormal{\Lambda}}

\newcommand{\Sig}{\mathnormal{\Sigma}}

\newcommand{\Om}{\mathnormal{\Omega}}

\newcommand{\N}{{\mathbb N}}
\newcommand{\Q}{{\mathbb Q}}
\newcommand{\R}{{\mathbb R}}

\newcommand{\Z}{{\mathbb Z}}

\newcommand{\PP}{{\mathbb P}}

\newcommand{\calA}{{\cal A}}
\newcommand{\calB}{{\cal B}}

\newcommand{\calD}{{\cal D}}

\newcommand{\calF}{{\cal F}}
\newcommand{\calG}{{\cal G}}

\newcommand{\calM}{{\cal M}}

\newcommand{\calP}{{\cal P}}
\newcommand{\calQ}{{\cal Q}}

\newcommand{\calS}{{\cal S}}

\newcommand{\calV}{{\cal V}}

\newcommand{\w}{\wedge}

\newcommand{\iy}{\infty}

\newcommand{\be}{\begin{equation}}
\newcommand{\ee}{\end{equation}}

\newcommand{\renyi}{R\'enyi }

\newcommand{\skp}{\vspace{\baselineskip}}
\newcommand{\noi}{\noindent}
\newcommand{\lan}{\langle}
\newcommand{\ran}{\rangle}
\newcommand{\lfl}{\lfloor}
\newcommand{\rfl}{\rfloor}
\newcommand{\lce}{\lceil}
\newcommand{\rce}{\rceil}

\newcommand{\eps}{{\varepsilon}}

\newcommand{\Hbar}{{\bar{H}}}

\newcommand{\Sbar}{{\bar{S}}}

\newcommand{\thetabar}{{\bar{\theta}}}

\newcommand{\xbar}{{\bar{x}}}

\newcommand{\Gmc}{{\mathcal{G}}}

\newcommand{\Rmb}{{\mathbb{R}}}

\newcommand{\mar}{\varsigma}
\newcommand{\PF}{\mathcal{P}\mathcal{F}}

\begin{document}

\title{Robust bounds and optimization at the large deviations scale for queueing models via R\'enyi divergence
}
\author{Rami Atar\thanks{Viterbi Faculty of Electrical Engineering, Technion}
\and Amarjit Budhiraja\thanks{Department of Statistics and Operations Research, University of North Carolina}
\and Paul Dupuis\thanks{Division of Applied Mathematics, Brown University}
\and Ruoyu Wu\thanks{Department of Mathematics, Iowa State University}}
\date{July 29, 2020}
\maketitle

\begin{abstract}
This paper develops tools to obtain robust probabilistic estimates
for queueing models at the large deviations (LD) scale.
These tools are based on the recently introduced robust R\'enyi bounds,
which provide LD estimates (and more generally risk-sensitive (RS) cost estimates)
that hold uniformly over an uncertainty class of models,
provided that the class is defined in terms of R\'enyi divergence
with respect to a reference model and that estimates are available for the reference model.
One very attractive quality of the approach is that the class to which the estimates
apply may consist of hard models, such as highly non-Markovian models
and ones for which the LD principle is not available.
Our treatment provides exact expressions as well as bounds
on the R\'enyi divergence rate on families of marked point processes,
including as a special case renewal processes.
%with respect to Poisson processes.
Another contribution is a general result that translates robust RS control problems,
where robustness is formulated via R\'enyi divergence,
to finite dimensional convex optimization problems,  when the control set is a finite dimensional convex set.
The implications to queueing are vast, as they apply in great generality.
This is demonstrated on two non-Markovian queueing models.
One is the multiclass single-server queue considered as a RS control problem,
with scheduling as the control process and exponential weighted queue length as cost.
The second is the many-server queue with reneging, with the probability of atypically large
reneging count as performance criterion. As far as LD analysis is concerned,
no robust estimates or non-Markovian treatment
were previously available for either of these models.

\skp

\noi

\skp

\noi
\end{abstract}

\section{Introduction}

\label{sec1}

An approach for obtaining robust
estimates on probabilistic models at the large deviations (LD) scale,
as well as on risk-sensitive (RS) functionals associated with
these models, has recently been proposed
based on \renyi divergence (RD) estimates \cite{atachodup}.
According to this approach, a family of models is considered that is defined
in terms of RD with respect to a reference model.
A tool, that we call in this paper
{\it robust \renyi bounds} (RRB), is then used to translate LD probability
estimates (and more generally, RS cost estimates) on the reference model
into ones which hold uniformly within this family.
This approach is particularly useful in cases when the reference model
is one that is easier to analyze than the collection of models on which
robust bounds are desired.
This paper applies these ideas to queueing models.

Indeed, queueing forms an ideal domain of applicability
of this approach,
for two main reasons. First, it is very often the case that
Markovian queueing models are considerably easier to handle than
non-Markovian ones. Among the many examples that strongly support this
assertion we mention
(1) the $M/M/n$ model for which the many-server law of large numbers (LLN)
limit is trivial as opposed to the $G/G/n$ counterpart
for which theory is involved and, in particular,
limit processes lie in the space of measure valued trajectories \cite{kasram11};
similarly, at the central limit theorem (CLT) scale, these two models
give rise to merely a diffusion on $\R$ \cite{halfin} and
a considerably more complicated, measure-valued diffusion \cite{kasram13}, respectively.
(2)
Queueing control problems, that in a Markovian setting can be analyzed
and solved as Markov decision processes (see numerous examples in \cite{sen09}),
but in a general setting, such as when service times are non-exponential,
require an infinite dimensional state descriptor and are less tractable.
Second, the robustness of estimates to perturbations
in the underlying distributions is important in applications.
Exponential service distribution (necessary for Markovity) is often assumed without
good statistical evidence or physical reasoning.
For example, a detailed statistical study
argues that there is a good fit of service time distributions in call centers
to lognormal \cite{brown05}, but there are
far more papers on many-server scaling limits, aimed at modelling large call centers,
in which servers operate with exponential distributions
than ones treating more general distributions.
In a much broader perspective, uncertainty in the underlying
distributions is a central issue in applying probabilistic queueing models to
real world systems.

To put LD estimates in a broader context of scaling limits
as far as sensitivity to perturbations in the underlying distributions is
concerned, it should be mentioned that most LLN and CLT results
in the queueing literature are tolerant to such perturbations
in the sense that the limits depend only on first or first and second
moments of the primitive data (the many-server limit regime
alluded to above is an exception). This has made these regimes
attractive for approximations and indeed provided motivation to study them.
On the other hand, the LD regime does not have obvious robustness
properties, as probabilities of rare events are sensitive to
the assumed tails of the primitives.
Consequently, model uncertainty issues and sensitivity to distributional
perturbations are much harder to deal with.
As already mentioned, this paper addresses these questions
by developing the approach of \cite{atachodup} in the context of queueing models.

The development in this paper, which involves performance measures that are determined by rare events and bounds defined in terms of \renyi divergence,
is analogous to prior work that bounds ordinary performance measures in terms of 
Kullback-Leibler divergence (also known as relative entropy).
This approach originated in a robust optimal control framework in \cite{dupjampet,petjamdup}, and was subsequently rediscovered a number of times and much developed in the literature \cite{iye,limsha,nilgha}.
The corresponding use in model uncertainty bounds and sensitivity bounds appeared later,
as in for example \cite{chodup,dupkatpanple}.

The literature on LD estimates for queueing models is rich.
A partial list of works dealing with non-Markovian queueing models
is as follows.
In \cite{ana2}, weak limit theorems are proved for the behavior of
a $G/G/1$ queue conditioned to exhibit an atypically large waiting time.
In \cite{gly94} the tail behavior of the waiting time steady state
distribution is identified for a large class of single server queues.
In \cite{maj06}, multiclass feedforward networks
are studied at the moderate deviations scale.
In \cite{puh07}, the LD principle (LDP) is established
and the rate function is identified for the generalized Jackson network.
See further references in these sources as well as in
the monographs \cite{gan04, shwwei, bud19} and the paper \cite{dupell3}
for numerous results on a variety of models in both Markovian and non-Markovian settings.

Sample path LDP of queueing models
are particularly difficult in network settings, due to the fact that
these models have discontinuous statistics.
References \cite{dupell3} \cite{gan04}, \cite{shwwei}, \cite{maj06},
\cite{puh07} do succeed is addressing such LDP.
Yet, even when tools such as LDP and formulas for the rate function are available,
a direct approach for obtaining estimates for an event of interest, uniformly
over a given family of models, may be notoriously hard,
as it amounts to solving a variational problem for each member in the family.
Unlike such a naive approach,
under the approach based on RRB, LD estimates have to be studied only for the reference model.
In fact, the approach does not even require that the LDP holds for each model in the family.

The rest of this paper is organized as follows.
The general approach that uses RRB to get robust estimates on families of models
is summarized in \S \ref{sec2}. In the same section,
an outline of the use of these bounds
for queueing models is provided, showing that for these models
the bounds heavily rely on estimating the \renyi divergence rate (an asymptotic
normalized version of the \renyi divergence) of a renewal process
with respect to a Poisson.
This provides a motivation to study such estimates for various families
of renewal processes. Results in this direction appear in \S \ref{sec3}.
Perhaps surprisingly, it seems that
such calculations have not appeared before in the literature.
In \S \ref{sec4} we provide a general development on RS control, and demonstrate it
with a queueing example.
Estimates on RS control are closely related to LD estimates, and in this section we argue
that they can be addressed by RRB.
A general result developed in \S\ref{sec41}
shows that a dramatic simplification occurs when RRB is used for RS control problems,
by which robust control estimates are transformed into  finite dimensional convex optimization problems  when the control set is a finite dimensional convex set.
In \S \ref{sec42}, we analyze a queueing control problem using this approach.
The model considered is the multi-class $G/G/1$ queue,
in which the control corresponds to scheduling jobs from the various classes.
As a reference model we use known RS control estimates for the multi-class $M/M/1$.
Finally, \S \ref{sec5} provides a queueing example for our robust approach to LD estimates.
The example consists of a queueing model with reneging.
Whereas reneging from queues is a very active research
field, little is known on LD estimates beyond the Markovian setting.
The robust LD estimates provided in this section are on both the
$G/G/1+G$ and the many-server $G/G/n+G$ models.
The reference model on which they rely is the $M/M/n+M$, for which
the sample path LDP has recently been developed in \cite{ABDW1}.

\section{Robust \renyi bounds}\beginsec
\label{sec2}

This section introduces the RRB and the approach that uses these bounds
to quantify robustness.
The RRB are described in \S \ref{sec21} and the form they take
under scaling is derived in \S \ref{sec22}.
In \S \ref{sec23} it is argued that in queueing applications
the \renyi divergence of a renewal process w.r.t.\ a Poisson
is key in the use of the approach, and
the notion of \renyi divergence rate is introduced.

\subsection{\renyi divergence}
\label{sec21}

Fix a measurable space $({\mathcal{S}},\mathcal{F})$ and denote by
$\mathcal{P}$ the set of probability measures on it. For $P,Q\in\mathcal{P}$,
the \textit{relative entropy} is given by
\[
R(Q\Vert P)=%
\begin{cases}
\displaystyle\int\log\frac{dQ}{dP}\,dQ & \text{if }Q\ll P\\
+\infty & \text{otherwise.}%
\end{cases}
\qquad
\]
Introduced in \cite{ren} (see \cite{lievaj} for a comprehensive treatment),
the {\it R\'enyi divergence} of degree $\alpha>1$, for
$P,Q\in\mathcal{P}$, is defined by
\[
R_{\alpha}(Q\Vert P)=%
\begin{cases}
\displaystyle\frac{1}{\alpha(\alpha-1)}\log\int\Big(\frac{dQ}{dP}%
\Big)^{\alpha}dP & \text{if }Q\ll P\\
+\infty & \text{otherwise.}%
\end{cases}
\]
For $\alpha=1$, one sets $R_{1}(Q\Vert P)=R(Q\Vert P)$. Whereas two different
formulas are used for the cases $\alpha=1$ and $\alpha>1$, it is a fact that
$\alpha\mapsto R_{\alpha}(Q\Vert P)$ is continuous on $[1,\alpha^{\ast}]$
provided $R_{\alpha^{\ast}}(Q\Vert P)<\infty$ for some $\alpha^{\ast}>1$. To
mention a few additional properties, one has that $\alpha\mapsto\alpha
R_{\alpha}$ is nondecreasing on $[1,\infty)$, and given $\alpha\geq1$, one
always has $R_{\alpha}(Q\Vert P)\geq0$, and $R_{\alpha}(Q\Vert P)=0$ if and
only if $Q=P$. A property that is of crucial importance in our use of
R\'{e}nyi divergence is its additivity for product measures, in the following
sense:
\begin{equation}
R_{\alpha}(Q_{1}\times Q_{2}\Vert P_{1}\times P_{2})=R_{\alpha}(Q_{1}\Vert
P_{1})+R_{\alpha}(Q_{2}\Vert P_{2}).\label{03}%
\end{equation}

It is well known that exponential integrals and relative entropy satisfy a
convex duality relation, stated as follows. Let $Q\in\mathcal{P}$. Then for
any bounded measurable $g:{\mathcal{S}}\rightarrow\mathbb{R}$,
\begin{equation}
\log\int e^{g}dQ=\sup_{P\in\mathcal{P}}\Big[\int gdP-R(P\Vert
Q)\Big].\label{01}%
\end{equation}
An analogous relation has been shown for R\'{e}nyi divergences
(\cite{atachodup}; related calculations first appeared in \cite{dvitod}).
Namely, fix $\alpha>1$. Then
\begin{equation}
\frac{1}{\alpha}\log\int e^{\alpha g}dQ=\sup_{P\in\mathcal{P}}\Big[\frac
{1}{\alpha-1}\log\int e^{(\alpha-1)g}dP-R_{\alpha}(P\Vert Q)\Big].\label{02}%
\end{equation}
The identity \eqref{02} may indeed be viewed as an extension of \eqref{01}, as
the latter is recovered by taking the formal limit $\alpha\downarrow1$ in the former.

Given $P,Q$ and $\alpha$, as well as an event $A\in\mathcal{F}$, it follows
from \eqref{02} by taking $g(x)=0$ [resp., $-M$] for $x\in A$ [resp., $x\in
A^{c}$] and sending $M\rightarrow\infty$, that
\begin{equation}%
\begin{split}
&  \frac{\alpha}{\alpha-1}\log P(A)-\alpha R_{\alpha}(P\Vert Q)\\
&  \quad\leq\log Q(A)\leq\frac{\alpha-1}{\alpha}\log P(A)+(\alpha-1)R_{\alpha
}(Q\Vert P)
\end{split}
\label{04}%
\end{equation}
(provided that $P(A)>0$ and $Q(A)>0$). 
The first inequality uses \eqref{02} as written, and the second reverses the roles of $P$ and $Q$.
In words: the logarithmic probability
of an event under $Q$ is estimated in terms of the same event under $P$ and
R\'{e}nyi divergence. It is also a fact that both inequalities in \eqref{04}
are tight, in the sense that given $\alpha$, $Q$ and $A$ one can find $P$ that
makes them hold as equalities (with different $P$ for each equality)
\cite{atachodup}.

The point of view of \cite{atachodup} is to regard \eqref{04} as perturbation bounds. Given a
nominal model $P$, \eqref{04} provides performance bounds on a true model $Q$
in terms of performance under $P$ and divergence terms.
The same is true in the more general case of a RS cost, namely
\begin{equation}
\begin{split}
\log\int e^{(\al-1)g}dQ\leq\frac{\alpha-1}{\alpha}\log \int e^{\al g}dP+(\alpha-1)R_{\alpha}(Q\Vert P).
\end{split}
\label{04+}
\end{equation}
In this paper we refer to \eqref{04+} and its special case \eqref{04}
as robust \renyi bounds (RRB).
If one fixes a reference model $P$ and a family $\calQ$ of true models
$Q$ defined by $\{Q:(\al-1)R_\al(Q\|P)<r\}$, some $r>0$, then for any $A$ \eqref{04}
gives  $\sup_{Q\in\calQ}\log Q(A)\le \frac{\al-1}{\al}\log P(A)+r$.
This expresses a uniform estimate on the performance under $Q$ in $\calQ$
in terms of that under $P$ and the size of the family (where the latter term
is interpreted in terms of \renyi divergence).
Clearly, an analogous statement can be made for RS cost
by appealing to \eqref{04+}, and similarly for lower bounds, by working with
$R_\al(P\|Q)$ instead of $R_\al(Q\|P)$.

\subsection{The RRB under scaling}
\label{sec22}

What makes the RRB particularly useful is that they
remain meaningful under standard LD scaling. We first demonstrate this in
a setting of IID random variables  (RVs), and then extend
to a continuous time setting.

\paragraph{IID data.}
Let $Z_{1},Z_{2},Z_{3}\ldots$ be a sequence of RVs, and let
$P$ and $Q$ be two probability measures that make this sequence IID. Let
$P_{n}$ and $Q_{n}$ denote the corresponding laws of $Z^{n}=(Z_{1},\ldots,Z_{n})$.
For each $n$, let $A_{n}$ be an event that is measurable on
$\sigma\{Z^{n}\}$, the $\sigma$-algebra generated by $Z^{n}$. We are
interested in
\[
\frac{1}{n}\log Q(A_{n}).
\]
By the IID assumption, we may appeal to \eqref{03}, according to which
$R_{\alpha}(Q_{n}\Vert P_{n})=nR_{\alpha}(Q_{1}\Vert P_{1})$. Thus by
\eqref{04} we obtain the bounds
\begin{equation}%
\begin{split}
&  \frac{\al}{\alpha-1}\frac{1}{n}\log P(A_n)-\al R_{\alpha}(P_{1}\Vert Q_{1})\\
&  \quad\leq\frac{1}{n}\log Q(A_n)\leq\frac{\alpha-1}{\al}\frac{1}{n}\log P(A_n)+(\alpha-1) R_{\alpha}(Q_{1}\Vert P_{1}).
\end{split}
\label{05}%
\end{equation}
In these bounds, the divergence terms remain of order 1 under scaling, and so
it is possible to compare the asymptotic behavior of $n^{-1}\log Q(A_n)$ to that of
$n^{-1}\log P(A_n)$. Moreover, while standard problems in the theory of LD are
concerned with limits of these expressions, we
emphasize that the bounds \eqref{05} are valid \textit{for all $n$.}

Regarding the normalized logarithmic probability as a performance measure in this setting
is indeed natural for studying probabilities of rare events. Thus our remark from
\S \ref{sec21} regarding uniform estimates on logarithmic probabilities
is relevant also for exponential decay rates. That is, given
$r>0$, let $\calQ$ consist of probability measures $Q$
under which $X_1,X_2,\ldots$ are IID and
$(\al-1)R_{\alpha}(Q_{1}\Vert P_{1})\leq r$. Then \eqref{05} gives
\[
\sup_{Q\in\calQ}\frac{1}{n}\log Q(A_n)
\leq\frac{\alpha-1}{\al}\frac{1}{n}\log P(A_n)+r.
\]
Again, a similar remark holds for RS cost, and
a lower bound is obtained similarly by working with $R_\al(P_1\|Q_1)$.

\paragraph{Beyond IID data.}
When the model is not based on an IID structure one can still apply
the RRB under scaling, but one must address the question
whether the normalized \renyi divergence term scales suitably.
Let us present this issue in a continuous time setting that better suits
the aims of this paper.
Let $\{Z_t,t\in\R_+\}$ be a stochastic process on the measurable space
$(\calS,\calF)$ and, thoughout this paper, for a general probability
measure $Q\in\calP$ denote $Q^Z_t=Q\circ Z|_{[0,t]}^{-1}$
(when there is no room for confusion, the dependence on the
process is omitted from the notation).
Then for any $t>0$, any event $A_t$ measurable on $\sig\{Z|_{[0,t]}\}$,
measure $P$ and collection of measures $\calQ$,
we have by \eqref{04}
\begin{equation}\label{09}
\sup_{Q\in\calQ}\frac{1}{t}\log Q(A_t)
\le
\frac{\al-1}{\al}\frac{1}{t}\log P(A_t)+(\al-1)\sup_{Q\in\calQ}\frac{1}{t}R_\al(Q^Z_t\|P^Z_t).
\end{equation}
If the last term remains bounded as $t\to\iy$
then one obtains uniform LD estimates within the family $\calQ$
by LD estimates on the reference model $P$ and the \renyi
divergence term. This method then remains effective in cases where
the latter term can be computed or estimated.

For statements that involve the limit $t\to\iy$, we shall need
a further piece of notation, used throughout.
Given a process $Z$ on $(\calS,\calF)$ and measures
$P,Q\in\calP$, the {\it \renyi divergence rate} (RDR) of $Q$ w.r.t.\ $P$ associated with
the process $Z$ is defined by
\[
r^Z_{\alpha}(Q\|P)
=\limsup_{t\to\iy}\frac{1}{t}R_{\alpha}(Q^Z_{t}\Vert P^Z_{t}).
\]
For a family $\calQ$ of probability measures, let
the RDR of $\calQ$ w.r.t.\ $P$ and of $P$ w.r.t.\ $\calQ$
be defined, respectively, by
\[
r^Z_\al(\calQ\|P)=\limsup_{t\to\iy}\sup_{Q\in\calQ}
\frac{1}{t}R_\al(Q^Z_t\|P^Z_t),
\qquad
r^Z_\al(P\|\calQ)=\limsup_{t\to\iy}\sup_{Q\in\calQ}
\frac{1}{t}R_\al(P^Z_t\|Q^Z_t).
\]
Again, the dependence on $Z$ will be omitted when there is
no confusion. With this notation, we have
\begin{equation}\label{091}
\limsup_{t\to\iy}\sup_{Q\in\calQ}\frac{1}{t}\log Q(A_t)
\le
\frac{\al-1}{\al}\limsup_{t\to\iy}\frac{1}{t}\log P(A_t)
+(\al-1)r^Z_\al(\calQ\|P).
\end{equation}

Often, the selection of $P$ with which to apply the bound
\eqref{09} or \eqref{091} is based on
considerations of tractability. If, for example, $P$ is a model under which performance
can be explicitly computed then one may use the approach in order to obtain
guaranteed bounds on a set of possibly intractable
models $Q$. Another consideration, that is especially relevant in engineering applications,
is that systems often operate under conditions that are
distinct from those they are designed for. For such systems, the bounds
provide guarantees on their true performance based on designed performance.

As a final general remark, given a particular event or a
sequence of events that are of interest, one can optimize over the parameter
$\alpha$ for the tightest upper and lower bounds. Namely, in both
\eqref{09} and \eqref{091} one may take the infimum over $\al>1$ on the
right hand side. This observation will be used in \S \ref{sec4}.

\subsection{Queueing models}\label{sec23}

Queueing models are described in terms of
service disciplines and stochastic primitives, where the latter term usually refers to
arrival processes, service times, routing and other processes.
The way in which we propose to use the RRB based approach
in the queueing context is by working with \renyi divergence estimates for the underlying
primitives rather than directly with the `state' processes that are used in describing
performance criteria (such as queue lengths, delay, idleness).
This is particularly natural when one views such models as dynamical
systems driven by renewal processes or more general counting processes
(or yet more generally, as marked counting processes).
To demonstrate this point, we provide two examples.

First, consider
the queue length process $X_t$ for a GI/GI/1 queue.
In this single server queue, arrivals follow a renewal process,
denoted by $A_t$, and service times are IID. Let
$S_t$ denote the potential service process: By the time the server
is busy for $t$ units of time, $S_t$ jobs have departed.
Assuming here, for simplicity, that at time zero
the server has no residual work, $S_t$ is also a renewal process.
The queue length satisfies the equations
\[
X_t=X_0+A_t-S(T_t),
\qquad\qquad
T_t=\int_0^t1_{\{X_s>0\}}ds.
\]
For our purpose, the key property is that
$X|_{[0,t]}$ is fully determined by its initial condition
and the primitives $A|_{[0,t]}$, $S|_{[0,t]}$ (this owes to the fact $T_t\le t$ for all $t$).
Hence, if such a queue is to be analyzed by comparison
to M/M/1, the relevant \renyi divergence term dictating
events measurable w.r.t.\ $X|_{[0,t]}$, is $R_\al(Q_t\|P_t)$,
where $Q_t$ is the law of $(A,S)|_{[0,t]}$ as a pair of (independent) renewal processes
and $P_t$ as a pair of Poisson processes.

Next consider a generalized Jackson network.
This is a network of $N$ service stations, each having an external
(possibly void) stream of arrivals, and upon departure from
a service station, jobs are routed probabilistically, according
to a given substochastic $N\times N$ matrix,
to one of the service stations or to leave the system.
Let $\{\xi_i(k)\}$ be $\{0,e_1,\ldots,e_N\}$-valued
RVs according to which these routings are determined:
$\xi_{ij}(k)=1$ dictates that the $k$th $i$-departure
is routed to station $j$.
All arrival processes and potential service processes
are assumed to be mutually independent renewals,
that are also independent of the routing
decision variables, $\xi$. Denote by $X_i$, $E_i$, $S_i$ and $D_i$
the queue length, external arrival, potential service and
departure processes, associated with service station $i$
for $1\le i\le N$. Let also $D_{ij}$ denote the counting process
of jobs departing from station $i$ and routed back to station $j$.
Finally, let $A_i$ denote the counting
process for total arrivals into station $i$, including
external arrivals and reroutings. Then the following
equations are satisfied:
\begin{align*}
X_i&=X_i(0)+A_i-D_i=X_i(0)+A_i-S_i\circ T_i\\
A_i&=E_i+\sum_jD_{ji} \hspace{4em} T_i=\int_0^\cdot1_{\{X_i>0\}}dt\\
D_{ij}&=R_{ij}\circ D_i \hspace{6em} R_{ij}=\sum_{k=1}^\cdot\xi_{ij}(k).
\end{align*}
Thus the dependence of queue length on the stochastic primitives
is far more complicated than in the case of a single node.
Yet, the key property alluded to above is valid in this complicated
scenario. That is, $X_i|{[0,t]}$, $1\le i\le N$, are dictated by their initial condition
and the primitives $E_i|_{[0,t]}$, $S_i|_{[0,t]}$ and
$(R_{ij}\circ S_i)|_{[0,t]}$, $1\le i,j\le N$.
A similar statement is valid for the busyness processes $T_i$,
the counting processes $D_{ij}$, etc.

The special case where $E_i$ and $S_i$ are Poisson
is referred to as a Jackson network. In this case,
the queue length is a Markov process with state space
$\Z_+^N$, and is far easier to analyze
than the non-Markovian model.
The perturbation that is required for translating results on
the Markovian model to the more general one again
has to do with a change of measure from a Poisson
to a renewal process. Once again, the perturbation can be expressed as a \renyi divergence term, this time for each $E_i$ and $S_i$, $1\le i\le N$.
A term that takes into account the routings $(R_{ij}\circ S_i)|_{[0,t]}$
is required only if perturbations of the routing matrix
are considered. This is important as far as robust
performance bounds are concerned, but note that it is
not required for turning a non-Markovian model into Markovian.
%In this paper we do not address perturbing the routing matrix.

As mentioned above,
the Markovian model is easier to handle than the non-Markovian one. As far as LD results are concerned,
the full LDP for the former was established
via a general approach by Dupuis and Ellis, as a special case of a large class of Markovian queueing models \cite{dupell2}.
Building on these results, the identification of the rate function was obtained
by \cite{atadup1} and by \cite{ign2}. Expressions for the rate function in these two references were provided
as a finite-dimensional convex optimization problem, and as a recursive formula, respectively. Denoting a rescaled version of $X$ by $X^n=n^{-1}X(n\cdot)$,
and letting $P$ stand for the probability measure that makes the primitives
$E_i$ and $S_i$ independent Poisson processes,
the LDP for the Jackson network provides an upper bound
in the form of a variational formula for the asymptotic expression
\[
\gamma(P,F)=\limsup_{n\to\iy}\frac{1}{n}\log P(X^n|_{[0,1]}\in F)
\]
where $F$ is any closed (in the $J_1$ topology) set of paths mapping $[0,1]$ to
$\R_+^N$. A typical set of interest is $F=\{\phi:\phi(t)  \in M \mbox{ for some } 
t \in [0,1]\}$, $M=\R_+^N-\prod_{i=1}^N[0,b_i)$,
expressing the buffer overflow event: one of the queues $X^n_i$ exceeds a threshold
$b_i$ some time during $[0,1]$, an event that is rare for large $n$ provided
that the network is stable.

LDP is known also for the generalized (that is, non-Markovian) Jackson network by \cite{puh07},
where the rate function is identified in terms of an optimization problem, that is not in general a convex
optimization problem, and for which a recursive formula such as \cite{ign2} is not available.
Considerably less is known on estimates at this scale which hold for a generalized Jackson network
{\it uniformly} w.r.t.\ the stochastic primitives within certain set.
However, \eqref{091} addresses precisely this question.

Indeed, to state the readily available corollary of \eqref{091}, set
\[
r_\al(\calQ\|P)=\sum_{i=1}^Nr_\al^{E_i}(\calQ\|P)+\sum_{i=1}^Nr_\al^{S_i}(\calQ\|P)
\]
to be the sum of RDR over all primitive processes.
Then given a collection $\calQ$,
we have the following uniform bound on generalized
Jackson networks associated with $Q\in\calQ$
in terms of performance of the Jackson network $P$, namely
\begin{equation}\label{010}
\sup_{Q\in\calQ}\gamma(Q,F)\le
\inf_{\al>1}\Big\{
\frac{\al-1}{\al}\gamma(P,F)
+(\al-1)r_\al(\calQ\|P)
\Big\}.
\end{equation}
As already mentioned, by the LDP, an upper bound
on $\gamma(P,F)$ is known, in the form of a variational formula.
Therefore the usefulness of \eqref{010}
depends on the ability to compute or provide an effective bound also on the
last term, that is, the RDR of a renewal process w.r.t.\ a Poisson.

%The key role played by the RDR, presented above for the Jackson network,
%is similarly played by it in relation to any queueing model that can be
%formulated as a dynamical system driven by
%renewal or other counting processes, which in the case of Poisson processes simplifies.

The case made above for the crucial importance of RDR estimates
for the applicability of the approach can be made in any scenario
where a queueing model is representable as a dynamical system
driven by renewal processes or other counting processes,
and in the special case of Poisson
driving processes is tractable (due to Markovity or for
any other reason).
Therefore the usefulness of studying the RDR in relation to the proposed approach is broad.
%Therefore the need to study the RDR in relation to the proposed approach is broad.

\section{Results on RDR}\beginsec
\label{sec3}

Calculations and bounds of entropy rate and R\'{e}nyi entropy rate
have been studied for some families of stochastic processes, including
Markov chains and hidden Markov models \cite{jac08}, \cite{ord04}.
However, the questions that
arise from the above discussion are concerned with the RDR of marked point processes with respect to marked Poisson point processes, also known as a Poisson random measure.
%
% counting processes
% (renewal or others) with respect to a Poisson process.
To the best of our knowledge, estimates on RDR for such models have not been studied
before. In this section we present some results in this direction.
The marks of the reference Poisson point process we consider will take values in some Polish space $S$ and will have iid distributions given by some probability measure $\mar$ on $(S, \calB(S))$, where $\calB(S)$ denotes the Borel $\sigma$-field on $S$. 
Denote by $\calM_F(S)$ the space of finite measures on $S$ equipped with the usual weak convergence topology.
A marked point process can be represented as a stochastic process $\{N_t\}$ with sample paths in $\Om = \calD([0,\infty): \calM_F(S))$. A rate $\lambda_0$ Poisson marked point process with mark distribution $\mar$ is a stochastic process $\{N_t\}$ such that (i) for all $0 \le s < t <\infty$ and  $A \in \calB(S)$,
$N_{t}(A)-N_{s}(A)$ is a Poisson random variable with mean $\la_0 (t-s)\mar(A)$; (ii) if for $k \in \N$,
 $0 \le s_i < t_i <\infty$ and  $A_i \in \calB(S)$, $i=1,2, \ldots k$, $\{(s_i, t_i]\times A_i\} \cap \{(s_j, t_j]\times A_j\} = \emptyset$, for all $1\le i < j \le k$, then the Poisson random variables 
$\{N_{t_i}(A_i) - N_{s_i}(A_i), 1 \le i \le k\}$ are mutually independent. We will refer to $\la_0 ds\times \mar(dz)$ as the intensity measure of such a marked Poisson point process.

Let  $\calF = \calB(\Om)$ and abusing notation, let $\{N_t\}_{t\ge 0}$ be the canonical coordinate process on $(\Om, \calF)$.
Fix $\la_0 \in (0,\infty)$ and $\mar \in \calP(S)$. Let $P$ be the unique probability measure on $(\Om, \calF)$ under which $N$ is a rate $\la_0$ marked Poisson process with  mark distribution $\mar$. 
We will consider the canonical filtration on $(\Om, \calF)$ which will be denoted as $\{\calF_t\}_{t\ge 0}$.

In this section we present two types of results: bounds on $r_\al(\calQ\|P)$
for $\calQ$ a family of models (namely probability measures on $(\Om, \calF)$) and on $r_\al(Q\|P)$ for a single model.
This is the content of Sections \ref{sec31} and  \ref{sec32}, respectively.
The proofs of the results stated in these two
sections appear in Appendix \ref{sec-a1} and \ref{sec-a2}, respectively.

\subsection{Bounds on RDR for families of processes}\label{sec31}

This section provides bounds on $r_\al(\calQ\|P)$ for families $\calQ$ of probability laws of counting processes
where as before $P$ is the probability law of a marked Poisson process with intensity measure $\la_0 ds\times \mar(dz)$.

For $x\ge0$ and $\al>1$ let $k_\al(x)$ denote the \renyi divergence
of order $\al$ of a Poisson RV with parameter $x \in (0,\infty)$ w.r.t.\ a Poisson RV with
parameter $1$. A direct computation gives
\begin{equation}\label{094}
k_\al(x)=\frac{x^\al-\al x+\al-1}{\al(\al-1)}.
\end{equation}
Note that for every $\al$, this function is nonnegative, strictly convex
and vanishes uniquely at $1$.

Denote by $\calV_0$ the set of mappings $v: \R_+\to\R_+$ such that $v(x)\to 0$ as $x\to \infty$.
Also, denote by $\PF$ the predictable $\sigma$-field on  $\Om \times \R_+$ and let
$\la: \Om \times \R_+ \times S \to (0,\infty)$ be a $\PF\otimes \calB(S)\backslash \calB(0,\infty)$  measurable map.  We will refer to such a map as a predictable process. 
We will consider probability measures $Q$ on $(\Om, \calF)$ under which $\{N_t\}$ is a marked point process
with intensity process $\la$. Such a probability measure can be characterized as the unique element of $\calP(\Om)$
under which for every bounded predictable process $u$ and $T<\infty$
$$E_P\left[\int_{[0,T]\times S} u(s,z) N(ds\,dz)\right] = E_P\left[\int_{[0,T]\times S} u(s,z) \la(s,z) ds \mar(dz) \right],$$
where we view $N$ as a $\calM_F([0,T]\times S)$-valued random variable which is defined by the relation $N((s,t]\times A)\doteq N_t(A)-N_s(A)$ for $0\le s<t\le T$, $A \in \calB(S)$.
We will be particularly interested in the case where $\la(s,x) = \hat \la(s)\psi(x)$, where
$\hat \la: \Om \times \R_+ \to (0,\infty)$ is a $\PF\backslash \calB(0,\infty)$-measurable map (also referred to as a predictable process) and $\psi: S \to (0,\infty)$ is a $\calB(S)\backslash\calB(0,\infty)$
measurable map satisfying $\int_S \psi(z) \mar(dz)=1$. This corresponds to the setting in which, under $Q$,
$\{N_t\}$ is a marked point process with points having iid distribution $\tilde \mar(dz)= \psi(z)\mar(dz)$.
In such a case, we will refer
to $(N,\la)$ as a marked Cox process.

In the special case where $S$ is a singleton $\{z^*\}$ (and so $\mar(dz)= \tilde \mar(dz) = \delta_{z^*}(dz)$),
$\{N_t\}$ is simply a Cox process with intensity process $\hat \la(\cdot)$ (see e.g. \cite[Section 1.1]{karr}).
In such a case we will occasionally also refer to $(N,\hat\la)$ as a Cox process. 
% Suppose that $\la: \Om \times \R_+ \to (0,\infty)$ is a predictable process. If under some probability measure $N$ is a Cox process  with  intensity process $\la(\cdot)$ , we will slightly abuse
% the usual terminology and say that $(N,\la)$ is a Cox process (under that measure).
For a probability measure $Q$ on $(\Om, \calF)$ and $t\in [0, \infty)$, let $Q_t^N \doteq Q \circ N_{[0,t]}^{-1}$ be the probability measure induced on
$\calD([0,t]: \calM_F(S))$ by the canonical coordinate process.
Motivation for the specific forms of the families $\calQ_i$ considered in the theorem appears after the theorem statement.

\begin{theorem}\label{th3}
	$\;$
\begin{enumerate}[(i)]
	\item  Fix $v\in\calV_0$ and $\al>1$. Consider the collection $\calQ_1$
of probability measures $Q$ on $(\Om, \calF)$ under which $(N,\la)$ 
is a marked point process
with intensity process $\la$ satisfying
\begin{equation}\label{095}
T^{-1}\int_0^T \int _S k_\al\left(\frac{\la(t,z)}{\la_0}\right) \mar(dz) dt\le u+v(T),
\qquad T>0,
\end{equation}
for some constant $u\ge 0$. Then
\[
r_\al(\calQ_1\|P)= \limsup_{t\to \infty} \frac{1}{t} \sup_{Q \in \calQ_1} R_{\alpha}(Q_t^N\|P_t^N) =u\la_0.
\]
\item Consider the collection $\calQ_2$ of probability measures $Q$ on $(\Om, \calF)$ under which
$(N,\la)$ is a  marked point process
with intensity process $\la$ satisfying
\begin{equation}\label{093}
a\le\frac{\la(\cdot)}{\la_0}\le b,
\end{equation}
for constants $0\le a\le 1\le b<\iy$. Then
\begin{equation}\label{201}
r_\al(\calQ_2\|P)=(k_\al(a)\vee k_\al(b))\la_0.
\end{equation}
As a special case, the identity holds for the family of measures
under which  the marks are iid (and independent of jump instants) with distribution $\psi(z)\mar(dz)$ and $\{N_t(S)\}$ is a delayed renewal processes with hazard rate $h$
(i.e., $(N, h(s)\psi(z))$ is a marked Cox process), and
 $a\le\frac{h(s)\psi(z)}{\la_0}\le b$ for all $(s,z)\in \R_+\times S$.
\\
\item  Let $v\in\calV_0$.
Consider the collection $\calQ_3$ of probability measures $Q$
under which $(N,\la)$  is a  marked point process
with intensity process $\la$ satisfying
\eqref{093} as well as
\begin{equation}\label{096}
\la_0-v(T)\le \frac{1}{T}\int_{[0,T]\times S}\la(t,z)\mar(dz) dt\le \la_0+v(T),\qquad T>0.
\end{equation}
Then
\[
r_\al(\calQ_3\|P)
=(pk_\al(a)+qk_\al(b))\la_0,
\]
where $p=\frac{b-1}{b-a}$ and $q=\frac{1-a}{b-a}$.
\\
\item
Let $v\in\calV_0$ and consider a collection $\calQ_4$ of probability measures under which
$(N,\la)$ is a  marked point process
with intensity process $\la$ satisfying \eqref{095} for some $\al=\al_0$,
as well as \eqref{096}.
Then for all $\al\in(1,\al_0)$,
\[
r_\al(\calQ_4\|P)=\frac{\la_0}{\bar\al}
\Big[\Big(\bar\al_0u+1\Big)^{\frac{\al-1}{\al_0-1}}-1\Big]
=(pk_\al(0)+qk_\al(c))\la_0,
\]
where $\bar\al=\al(\al-1)$, $\bar\al_0=\al_0(\al_0-1)$,
$p=1-q$, $q=(\bar\al_0u+1)^{-\frac{1}{\al_0-1}}$ and $c=(\bar\al_0u+1)^{\frac{1}{\al_0-1}}$.
\end{enumerate}
\end{theorem}

The proof appears in \S \ref{sec-a1}.
Figure \ref{fig1} provides several numerical evaluations of the RDR for families analyzed in Theorem \ref{th3}.

\begin{remark}
Items (i) and (ii)
of the result are concerned with classes of marked point processes
for which a certain constraint is put on the size of perturbation of the stochastic intensity.
Note that for $\al=2$, the left hand side of \eqref{095} gives half the second moment
centered about 1 of the empirical distribution 
$$\frac{1}{T} \int_{[0,T]\times S} \delta_{\la(t,z)/\la_0} \mar(dz) dt.$$
For other values of $\al$, it provides different types of  level dispersion about 1 that take the form of higher order moments.
In the same vein, \eqref{093} can be seen as a constraint on the $L_\iy$ norm of
of the same empirical distribution, centered about $\frac{a+b}{2}$.

The motivation behind parts (iii) and (iv) is that the reference value $\la_0$
may play an additional role. If a parameter is regarded as a first order approximation,
it may often mean that over a long period of time it represents the true average.
%Thus, for example, part (iv) with $\al_0=2$ says with a certain degree of certainty,
%the empirical mean intensity over a long time is given by
%$\la_0$, and in addition, the empirical second moment centered about 1 is bounded
%by a certain constant.
Clearly, this additional constraint makes the class of models smaller than the classes
from items (i) and (ii), and leads to tighter bounds.
\end{remark}

\begin{figure}
\center{
\includegraphics[width=0.48\textwidth,height=0.4\textwidth]{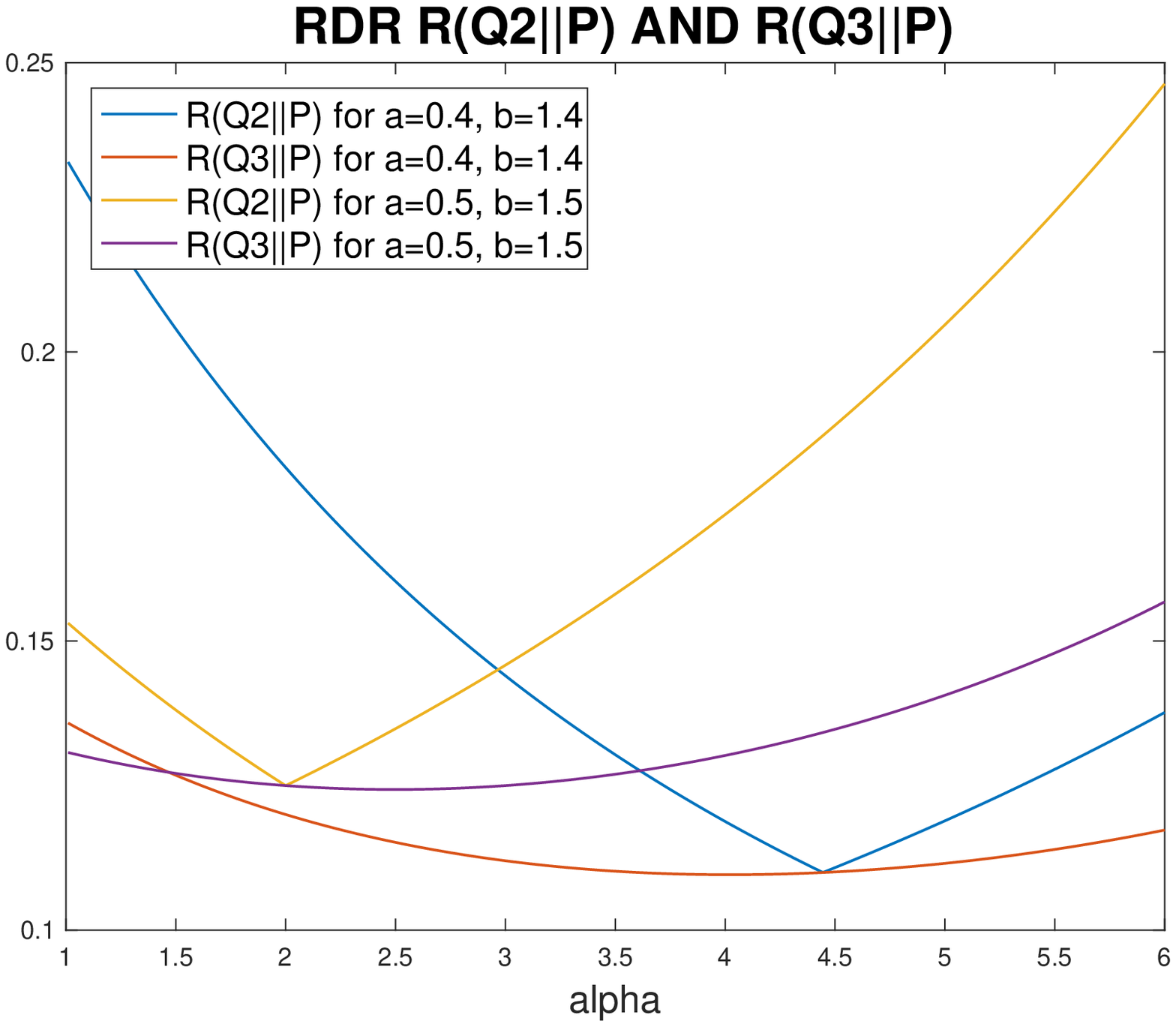}
\includegraphics[width=0.48\textwidth,height=0.4\textwidth]{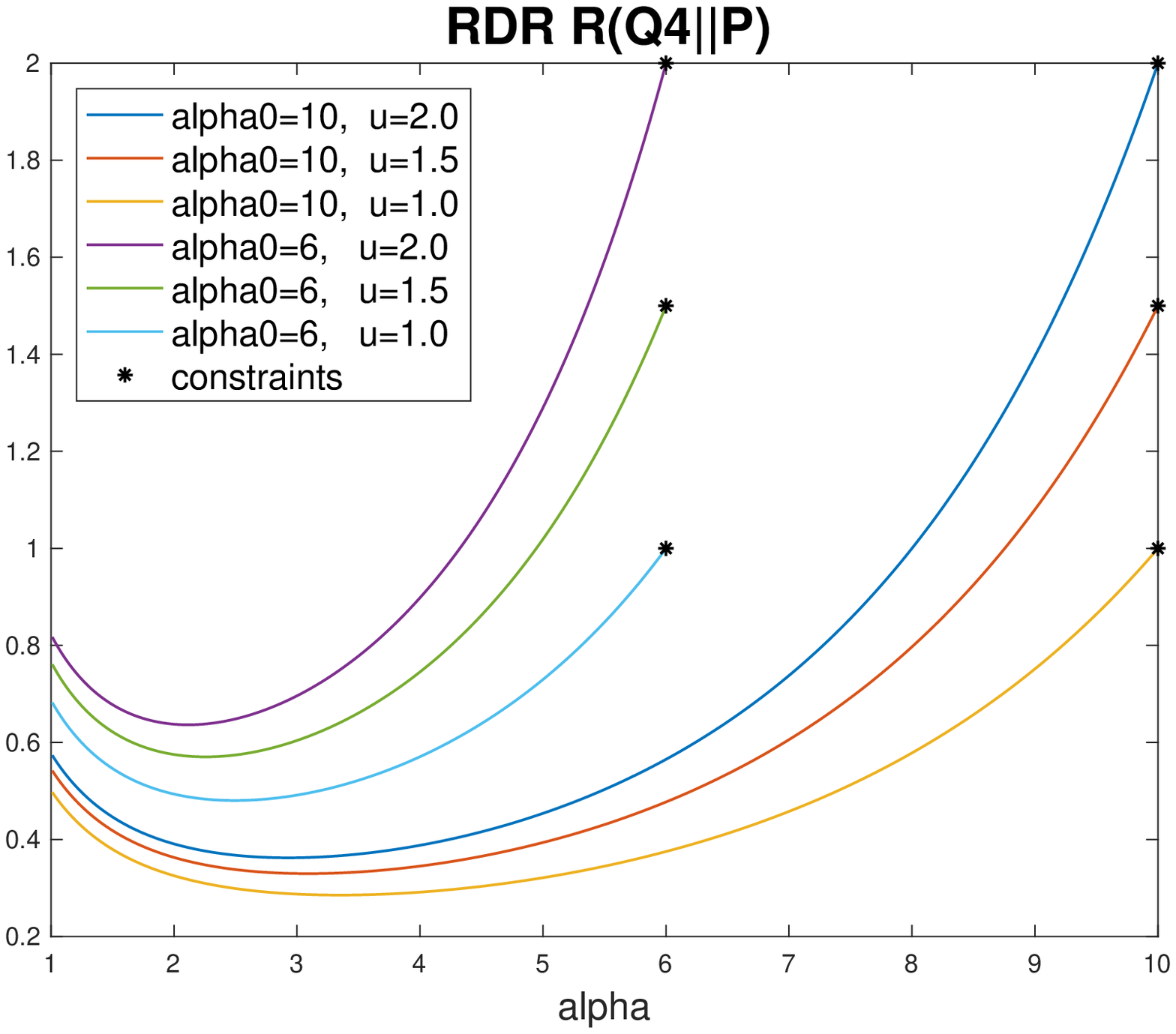}
}
\caption{RDR for various families as a function of $\al$. Left: $r_\al(\calQ_2\|P)$
and $r_\al(\calQ_3\|P)$ for different values of $a$ and $b$.
Right: $r_\al(\calQ_4\|P)$ for different constraint pairs $(\al_0,u)$.
In all cases, $\la_0=1$.}
\label{fig1}
\end{figure}

\subsection{Bounds on RDR for a single renewal process}\label{sec32}

Recall that $P$ is the unique probability measure on $(\Om, \calF)$ under which the canonical coordinate process $N$ is a marked Poisson process
with rate $\la_0$ and mark distribution $\mar$. In this section, for simplicity, we take $\la_0=1$.
Let  $Q$ be another probability measure on $(\Om, \calF)$ under which $N$ is a marked renewal process with mark distribution $\tilde \mar(dz) = \psi(z)\mar(dz)$
and inter-jump distribution $\pi$. Note that in such a process the collection of jump-instants is independent of the collection of marks, and inter-jump times and marks are iid. We denote such a process as a $(\pi, \psi)$-marked renewal process.
Assume that $\pi$ has a density denoted by $g$, and let $h$ denote the hazard
rate, $h(x)\doteq g(x)/\pi\lbrack x,\infty)$, with $h(x)=0$ if $\pi[x,\iy)=0$. Define, for $x\ge0$,
\begin{equation}
	\label{eq:H}
	H(x)\doteq\int_{0}^{x}(1-h(s))ds+\log h(x)=x+\log g(x),	
\end{equation}
with $H(x)=-\iy$ when $g(x)=0$.

To state the next result, let
\begin{align*}
	\gamma(s) & \doteq \int e^{sH(y)} \nu(dy), \quad s \in \Rmb, \\
	\hat\gamma(p,q,\al) & \doteq \frac{\gamma(q\al)^{p/q}-1}{p}, \quad p,q \ge 1,
\end{align*}
where $\nu$ is the standard exponential distribution.
Denote 
\begin{equation*}
	G_\alpha^{(1)} \doteq \inf_{p,q \ge 1:p^{-1}+q^{-1}=1}\hat\gamma(p,q,\al).
\end{equation*}
Also, let
\[
\beta(\la)\doteq\log \int e^{\la_1y+\la_2H(y)} \nu(dy),\qquad \la=(\la_1,\la_2)\in\R^2,
\]
and let $\beta^*$ be the Legendre-Fenchel transform:
\[
\beta^*(x)=\sup_{\la}\{\lan\la,x\ran-\beta(\la)\},\qquad x\in\R^2.
\]
For $\theta\in(0,\iy)$, denote
% \begin{align*}
% 	G_\al^{(2)}(\theta) & \doteq \theta \left\{\sup_{x\in\R^2:x_1\le\theta^{-1}} \left[ \alpha x_2 - \beta^*(x) \right]
% 	\wedge \sup_{x\in\R^2:x_1\ge\theta^{-1}} \left[ \alpha x_2 - \beta^*(x) \right]\right\}, \\
% 	G_\al^{(3)}(\theta) & \doteq \theta \sup_{x_2\in\R:} \left[ \alpha x_2 - \beta^*(\theta^{-1},x_2) \right].
% \end{align*}
\begin{align*}
	G_\al^{(2)}(\theta) & \doteq \theta \sup_{x\in\R^2:x_1\le\theta^{-1}} \left[ \alpha x_2 - \beta^*(x) \right], \\
	G_\al^{(3)}(\theta) & \doteq \theta \sup_{x_2\in\R} \left[ \alpha x_2 - \beta^*(\theta^{-1},x_2) \right].
\end{align*}
Recall that
$r_\al^N(Q\|P)= \limsup_{t\to \infty} \frac{1}{t}  R_{\alpha}(Q_t^N\|P_t^N)$. For $z \in \R$, we denote  $z^+=0\vee z$. 
Then we have the following upper bounds.
\begin{theorem}\label{thm1}
	Assume that $\bar H\doteq\sup_{x\in \R_+} H(x)<\iy$. Also suppose that
	$c(\al) \doteq \int_S (\psi^{\al}(z)-1)\mar(dz)<\infty$.
	Then the following hold for $\alpha > 1$.
	\begin{itemize}
	\item[(a)]
		\begin{equation}\label{eq:rough_bd}
		r_\al^N(Q\|P)\le\frac{e^{\al\bar H}-1 + c(\al)}{\al(\al-1)}.
		\end{equation}
	\item[(b)]
		\begin{equation}\label{eq:renew_1}
			r_\al^N(Q\|P)
			\le\frac{[G_\al^{(1)}]^++ c(\al)}{\al(\al-1)}.
		\end{equation}
	\item[(c)]
		If $\beta$ is finite in a neighborhood of the origin, then
		\begin{equation}\label{eq:renew_2}
		r_\al^N(Q\|P)
		\le\frac{[\sup_{\theta\in(0,\iy)}G_\al^{(2)}(\theta)]^+ + c(\al)}{\al(\al-1)}.
		\end{equation}	
	\item[(d)]
		If $\beta$ is finite in a neighborhood of the origin and $\gamma(s) < \infty$ for all $s \le 0$, then
		\begin{equation}\label{eq:renew_3}
		r_\al^N(Q\|P)
		\le\frac{[\sup_{\theta\in(0,\iy)}G_\al^{(3)}(\theta)]^+ + c(\al)}{\al(\al-1)}.
		\end{equation}	
	\end{itemize}
\end{theorem}

The proof of this result appears in \S \ref{sec-a2}.

% We would like to do the following:
%
% - Try to relax the assumptions $\bar H<\iy$ and/or $\uu{h}>0$.
%
% - Be able to calculate specific cases
%
% - Attempt at a lower bound

\skp

\begin{remark}
	We now make some comments on the assumption and behavior of different bounds in Theorem \ref{thm1}.
	\begin{itemize}
	\item
		For \eqref{eq:renew_2} and \eqref{eq:renew_3}, the assumption that $\beta$ is finite in a neighborhood of the origin is needed to apply the strengthened Cram\'{e}r's theorem \cite[Corollary 6.1.6]{demzei} and Varadhan's integral lemma \cite[Lemma 4.3.6]{demzei}.
		Unfortunately, if the support of $\pi$ is not $\Rmb_+$, then $H$ is $-\infty$ at some place and $\beta$ is not always finite around the origin.

	\item
		For \eqref{eq:renew_3}, the assumption that $\gamma(s) < \infty$ for all $s \le 0$ (together with the requirement that $\bar H <\infty$) rules out the possibility that $\pi$ is Exponential$(\rho)$ for $\rho \ne 1$.
	\item
	Since $G_\al^{(3)}(\theta) \le G_\al^{(2)}(\theta)$,	the bound in \eqref{eq:renew_3} is clearly better than the bound in \eqref{eq:renew_2} (though the former requires stronger
		assumptions).
Also, the bounds in \eqref{eq:renew_1}, \eqref{eq:renew_2} and \eqref{eq:renew_3} are all better than the rough bound \eqref{eq:rough_bd}. This can be seen as follows.
		
		For \eqref{eq:renew_1}, since $\gamma(s) \le e^{s\Hbar}$, we have $\hat\gamma(p,q,\alpha) \le \frac{e^{p\alpha\Hbar}-1}{p}$.
		Taking $p \to 1$ gives $G_\alpha^{(1)} \le e^{\alpha\Hbar}-1$.
		
		For \eqref{eq:renew_2}, note that for fixed $\theta \in (0,\infty)$,
		\begin{align}
			G_\alpha^{(2)}(\theta) & = \theta \sup_{x\in\R^2:0 \le x_1\le\theta^{-1}} \left[ \alpha x_2 - \beta^*(x) \right] \notag \\
			& = \theta \sup_{x\in\R^2:0 \le x_1\le\theta^{-1}} \left[ \alpha x_2 - \sup_{\lambda \in \R^2} \left\{ \lan \lambda,  x\ran - \beta(\lambda) \right\} \right] \notag \\
			& = \theta \sup_{x\in\R^2:0 \le x_1\le\theta^{-1}} \inf_{\lambda \in \R^2 } \left[ \alpha x_2 -  \lambda_1 x_1 - \lambda_2 x_2 + \beta(\lambda) \right] \notag \\
			& \le \theta \inf_{\lambda \in \R^2} \sup_{x\in\R^2:0 \le x_1\le\theta^{-1}} \left[ \alpha x_2 -  \lambda_1 x_1 - \lambda_2 x_2 + \beta(\lambda) \right] \notag \\
			& = \theta \inf_{\lambda_1 \in \Rmb} \sup_{0 \le x_1\le\theta^{-1}} \left[ - \lambda_1 x_1 + \beta(\lambda_1,\alpha) \right], \notag\\
			&= \inf_{\lambda_1 \in \Rmb} \left[ (\lambda_1)^{-}  + \theta\beta(\lambda_1,\alpha) \right]
			 \label{eq:rmk_1}
		\end{align}
		where the fifth line follows on observing that $\sup_{x_2 \in \Rmb} \left[ \alpha x_2 -  \lambda_1 x_1 - \lambda_2 x_2 + \beta(\lambda) \right] = \infty$ when $\lambda_2 \ne \alpha$.
		If $0 < \theta < 1$, taking $\lambda_1=0$ in \eqref{eq:rmk_1} we have
		\begin{equation*}
			G_\alpha^{(2)}(\theta) \le \theta \beta(0,\alpha) \le \alpha \Hbar \le e^{\alpha\Hbar}-1,
		\end{equation*}
		If $\theta \ge 1$, taking $\lambda_1 = 1-\theta \le 0$ in \eqref{eq:rmk_1} we have
		\begin{align*}
			G_\alpha^{(2)}(\theta) & \le -\lambda_1 + \theta \beta(\lambda_1,\alpha) = \theta-1 + \theta \log E_P[e^{(1-\theta) \Del + \alpha H(\Del)}] \\
			& \le \theta-1 + \alpha \Hbar \theta - \theta \log \theta \le e^{\alpha\Hbar}-1,
		\end{align*}
		where the last inequality becomes equality when $\theta = e^{\alpha \Hbar}$.
		Therefore $\sup_{\theta \in (0,\infty)} G_\alpha^{(2)}(\theta) \le e^{\alpha\Hbar}-1$ and \eqref{eq:renew_2} is better than \eqref{eq:rough_bd}.
		Finally, since  \eqref{eq:renew_3} is better than \eqref{eq:renew_2}, it is also better than \eqref{eq:rough_bd}.	
	\end{itemize} 
\end{remark}

\subsection{Examples}\label{sec:exam}
We now consider a few specific cases of $\pi$. For simplicity these examples are concerned with
point processes without marks (namely the case where $\mar(dz)= \delta_{z^*}(dz)$ and hence $c(\alpha)=0$ in Theorem \ref{thm1}).
The first example is that of an exponential distribution.

\begin{example}
\label{exa:exp}
	Suppose $\pi=$ Exp$(\rho)$ with rate $\rho>1$, namely $g(x) = \rho e^{-\rho x}$.
	It turns out that in this case the right sides of \eqref{eq:renew_2} and \eqref{eq:renew_3} are the same and in fact the inequalities in both cases can be replaced by equalities.  Note that in this example $\gamma(s)=\infty$ for $s \le -\frac{1}{\rho-1}$, which violates the assumption required for \eqref{eq:renew_3} in part (d).  Actually in \eqref{eq:renew_2} and \eqref{eq:renew_3} the inequality can be changed to equality even for the case  $\rho \in (0,1]$.  However, we note that $H(x) = -(\rho-1)x + \log \rho \to \infty$ as $x \to \infty$ when $\rho \in (0,1]$, which violates the assumption $\Hbar<\infty$ required for Theorem \ref{thm1}.
	Proofs of the above statements are given in Appendix \ref{sec-a3}
	This example shows that the conditions assumed in Theorem \ref{thm1} are  not essential for the result.
	%It would be nice if we can relax the assumption on $H$. 
	\qed	
\end{example}

The second example is Gamma$(k,\rho)$.

\begin{example}\label{ex:32}
	Suppose $\pi=$ Gamma$(k,\rho)$ with $k \ge 1$ and $\rho>1$, namely $g(x) = \frac{\rho^k}{\Gam(k)} x^{k-1} e^{-\rho x}$, for $x\ge 0$.
	For this example, computing an explicit expression for the \renyi divergence is harder and thus we will make use of the bounds in Theorem \ref{thm1}.
	%Using Lemma \ref{lem:relax_assump}, we don't have to worry about assumptions on $h$ in Theorem \ref{thm1}.
	Since $\rho>1$ and $k \ge 1$, $e^{H(x)} = g(x)e^x$ is bounded from above.
	Also for $\la=(\lambda_1,\lambda_2)$ in a sufficiently small neighborhood of  the origin,
	\begin{align*}
		\beta(\la) & = \log \int e^{\lambda_1 x} g(x)^{\lambda_2} e^{\lambda_2 x} e^{-x} \, dx \\
		& = \log \int \left( \frac{\rho^k}{\Gam(k)} \right)^{\lambda_2} x^{\lambda_2(k-1)} e^{-(\lambda_2\rho-\lambda_2-\lambda_1+1) x} \,dx \\
		& = \log \left[ \left( \frac{\rho^k}{\Gam(k)} \right)^{\lambda_2} \frac{\Gam(1+\lambda_2(k-1))}{(1+\lambda_2(\rho-1)-\lambda_1)^{1+\lambda_2(k-1)}} \right] < \infty.
	\end{align*}
	So all assumptions for \eqref{eq:renew_2} hold. Note however that assumptions for \eqref{eq:renew_3} are not satisfied since $\gamma(s)=\infty$ for 
	$s \le -(\rho-1)^{-1}$.
	Using Theorem \ref{thm1}(c) one can give the following 
	explicit bound for the \renyi divergence rate by estimating
$\sup_{\theta \in (0,\infty)} G_\alpha^{(2)}(\theta)$. 
	\begin{equation} \label{eq:bdforgam} r_\al^N(Q\|P) \le \frac{1}{\alpha(\alpha-1)} \sup_{\theta \in (0,\infty)} G_\alpha^{(2)}(\theta) \le \frac{1}{\alpha(\alpha-1)} \left[ \left( \frac{\Gam(1+\alpha(k-1))}{(\Gam(k))^\alpha} \rho^{\alpha k} \right)^{\frac{1}{1+\alpha(k-1)}} -\alpha(\rho-1) - 1 \right].\end{equation}
	Details of this calculation are given in Appendix \ref{sec-a3}.
	When $k=1$, namely when $\pi$ is Exp$(\rho)$, the   bound on the right side equals
	$\rho^\alpha -\alpha(\rho-1) - 1$ and the inequalities in the above display are in fact equalities.	
	\qed
\end{example}

The next example can be used to obtain RDR bounds for certain types of phase-type distributions. 
%Next is a lemma on estimating the bound \eqref{eq:renew_2}, which will be used to study the example of certain phase-type distributions.

\begin{example}
	\label{lem:convert_to_exp}
	Suppose the density $g(x) \le Ce^{-\sigma x}$ for some $\sigma > 1$.
	In this case, once again, the assumptions for \eqref{eq:renew_3} are not satisfied in general. However as we check below the assumptions for
	\eqref{eq:renew_2} hold.
	First note that $e^{H(x)} = g(x)e^x \le C e^{-(\sigma-1)x}$ and since $\sigma >1$, $\bar H <\infty$.
	
	Next note that for $\la=(\lambda_1,\lambda_2)$ such that $1+\lambda_2(\sigma-1)-\lambda_1 > 0$, we have
	\begin{align*}
		\beta(\la) & = \log \int e^{\lambda_1 x} g(x)^{\lambda_2} e^{\lambda_2 x} e^{-x} \, dx \\
		& \le \log \int C^{\lambda_2} e^{-(\lambda_2\sigma-\lambda_2-\lambda_1+1) x} \, dx \\
		& = \lambda_2 \log C + \log  \frac{1}{1+\lambda_2(\sigma-1)-\lambda_1} < \infty.
	\end{align*}
	Thus we have verified that all the asuumptions needed for \eqref{eq:renew_2} are satisfied. 
	Using Theorem \ref{thm1}(c) one can give the following  simple form bound for the quantity
	on the right side of \eqref{eq:renew_2}.
	\begin{equation}\label{eq:bdfphas} r_\al^N(Q\|P) \le\frac{1}{\alpha(\alpha-1)}\sup_{\theta\in(0,\iy)}G_\al^{(2)}(\theta)^+ \le \frac{1}{\alpha(\alpha-1)} \left[ C^\alpha - 1 - \alpha(\sigma-1) \right].\end{equation}
	Details of this calculation are given in Appendix \ref{sec-a3}.
	\qed
\end{example}

\section{Robust control of tail properties for a scheduling problem}
\beginsec
\label{sec4}

When considering ordinary cost structures the variational
representation for exponential integrals in terms of relative entropy is the
starting point for a formulation of optimization and control design that is
robust with respect to model errors, where errors are measured by relative
entropy distances. To be precise, one can formulate problems such that their
solution gives the tightest possible bounds on a given performance measure
for a family of models, where the family is defined by a relative entropy
distance to a design model \cite{petjamdup,petugrsav}. Alternatively, one can
fix a desired performance bound, and find the control which gives the
largest possible family of models across which the performance criteria is
guaranteed to hold. In this section we investigate analogous situations
where in place of ordinary cost structures we use costs that are determined
by rare events, i.e., risk-sensitive costs. Let 
$({\mathcal{S}},\mathcal{F})$ and 
$\mathcal{P}$ be as in Section \ref{sec21}.

\subsection{A general approach}
\label{sec41}
Let $g:\mathcal{S}\rightarrow \mathbb{R}$ be bounded and measurable.
From the identity 
\begin{equation*}
\frac{1}{\alpha }\log \int e^{\alpha g}dP=\sup_{Q\in \mathcal{P}}\left[ 
\frac{1}{\alpha -1}\log \int e^{(\alpha -1)g}dQ-R_{\alpha }(Q\Vert P)\right]
,
\end{equation*}%
valid for $\alpha >1$, one can easily obtain, for all $0<\beta <\gamma $, 
\begin{equation}
\frac{1}{\gamma }\log \int e^{\gamma g}dP=\sup_{Q\in \mathcal{P}}\left[ 
\frac{1}{\beta }\log \int e^{\beta g}dQ-\frac{1}{\gamma -\beta }R_{\frac{%
\gamma }{\gamma -\beta }}(Q\Vert P)\right]   \label{305}
\end{equation}%
(extensions to unbounded $g$ are also possible). Fix some risk sensitivity parameter $\beta $ and a class of models $\mathcal{%
Q}$. Let 
\begin{equation}
f(\alpha )=\sup \{R_{\alpha }(Q\Vert P):Q\in \mathcal{Q}\},\qquad \alpha \in
(1,\infty ).  \label{eqn:defoff}
\end{equation}%
By appealing to \eqref{305} we can show the following.

\begin{theorem}
\label{thm2} Fix $\beta >0$, $g$, $P$, $\mathcal{Q}$ and $f$ as above.
Then 
\begin{equation*}
\sup_{Q\in \mathcal{Q}}\frac{1}{\beta }\log E_{Q}e^{\beta g}\leq
\inf_{\gamma >\beta }F(\beta ,\gamma ),\qquad \text{where}\qquad F(\beta
,\gamma )=\left[ \frac{f(\frac{\gamma }{\gamma -\beta })}{\gamma -\beta }+%
\frac{1}{\gamma }\log E_{P}e^{\gamma g}\right] .
\end{equation*}%
Moreover, $\tilde{\gamma}\mapsto F(\beta ,1/\tilde{\gamma})$ is a convex
function.
\end{theorem}

We propose to use Theorem \ref{thm2} as the basis for the formulation of
optimization and control problems, so that given a class $\mathcal{Q}$, an
upper bound is obtained across $\mathcal{Q}$ for a RS control (or
optimization) problem. For certain problems we expect to be able to say
more, which is that the bound is tight in some sense. A controlled process $X
$ is considered with cost of the form $E[e^{\beta g(X)}]$, where $\beta $ is
the sensitivity parameter. The RS control problem will be concerned with $%
E^{a}[e^{\beta g(X)}]$, where $a$ denotes a control or a parameter to be
optimized over, and the robust version of this problem is one where a
control $a$ is sought to minimize the RS cost uniformly in the family of
models. In this context, Theorem \ref{thm2} gives 
\begin{equation}\label{306}
\inf_{a}\sup_{Q\in \mathcal{Q}}\frac{1}{\beta }\log E_{Q}^{a}e^{\beta g}\leq
\inf_{a}\inf_{\gamma >\beta }\left[ \frac{f(\frac{\gamma }{\gamma -\beta })}{%
\gamma -\beta }+\frac{1}{\gamma }\log E_{P}^{a}e^{\gamma g}\right] .
\end{equation}%
Two highly attractive aspects of this bound are

\begin{itemize}
\item[(i)] it turns an $\infty$-dimensional game into a finite dimensional
minimization problem when $a$ is finite dimensional;

\item[(ii)] the minimization over $\gamma$ is tractable computationally,
thanks to the convexity stated in Theorem \ref{thm2}.
\end{itemize}

\noindent With regard to the optimization over $a$, that is of course
related to the structure of the particular problem. However, it is worth
noting that the difficulty of this problem is often related to the
difficulty of the ordinary analogue, i.e., $\inf_{a}E_{P}^{a}g$. For the
example from queueing presented below we see that the risk-sensitive
optimization problem has a structure that is very similar to that of the
ordinary analogue. The function $f$ is the element that distinguishes this
problem from its relative entropy/ordinary cost analogue, for which $f$ is
essentially a constant. In some sense, $f$ captures the critical,
distribution dependent properties of the tail behavior of $\mathcal{Q}$.
The only part of Theorem \ref{thm2} that does not follow directly from 
\eqref{305} is the last sentence, 
which we now address.

\begin{lemma}
\label{lem:convex_H} Let $X$ be a non-zero non-negative random variable.
Then the function 
\begin{equation*}
m(\theta )\doteq \theta \log EX^{1/\theta }
\end{equation*}%
is convex in $\theta >0$.
\end{lemma}

\begin{proof}
	It suffices to show that
	\begin{equation*}
		m(\theta) \le \lambda m(\theta_1) + (1-\lambda) m(\theta_2)
	\end{equation*}
	for every $\theta_1,\theta_2 \in (0,\infty)$ and $\lambda \in (0,1)$, where $\theta \doteq \lambda \theta_1 + (1-\lambda) \theta_2$.
	Assume without loss of generality that $m(\theta_1) < \infty$, $m(\theta_2) < \infty$.
	Applying H\"{o}lder's inequality with $p=\frac{\theta}{\lambda\theta_1}$ and $q=\frac{\theta}{(1-\lambda)\theta_2}$, we have
	\begin{align*}
		m(\theta) & = \theta \log E[X^{\frac{\lambda}{\theta}} X^{\frac{1-\lambda}{\theta}}] \le \theta \log \left[ \left( EX^{\frac{\lambda}{\theta}p} \right)^{1/p} \left( EX^{\frac{1-\lambda}{\theta}q} \right)^{1/q} \right] \\
		& = \frac{\theta}{p} \log EX^{1/\theta_1} + \frac{\theta}{q} \log EX^{1/\theta_2} = \lambda m(\theta_1) + (1-\lambda) m(\theta_2).
	\end{align*}
	This completes the proof.
	\qed
\end{proof}

\begin{lemma}
\label{lem:convex_bd} For $0<\beta <\gamma$, let
\begin{equation*}
h(\gamma )\doteq \frac{f(\frac{\gamma }{\gamma -\beta })}{\gamma -\beta }+%
\frac{1}{\gamma }\log E_{P}e^{\gamma g},
\end{equation*}%
where $f(\alpha )$ is defined by (\ref{eqn:defoff}). Then the function $%
\tilde{h}(\tilde{\gamma})\doteq h(1/\tilde{\gamma})$ is convex in $\tilde{%
\gamma}\in (0,1/\beta )$, i.e, $h(\gamma )$ is convex in $1/\gamma $.
\end{lemma}

\begin{proof}
Since 
\begin{equation*}
R_{\alpha }(Q\Vert P)=\frac{1}{\alpha (\alpha -1)}\log \int \left( \frac{dQ}{%
dP}\right) ^{\alpha }dP,
\end{equation*}%
we can write 
\begin{align*}
\frac{1}{\gamma -\beta }R_{\frac{\gamma }{\gamma -\beta }}(Q\Vert P)& =\frac{%
1}{\gamma -\beta }\frac{1}{\frac{\gamma }{\gamma -\beta }(\frac{\gamma }{%
\gamma -\beta }-1)}\log \int \left( \frac{dQ}{dP}\right) ^{\frac{\gamma }{%
\gamma -\beta }}dP \\
& =\frac{1}{\beta }\frac{1}{\frac{\gamma }{\gamma -\beta }}\log \int \left( 
\frac{dQ}{dP}\right) ^{\frac{\gamma }{\gamma -\beta }}dP.
\end{align*}%
Therefore
\begin{equation*}
\tilde{h}(\tilde{\gamma})=h(1/\tilde{\gamma})=\sup_{Q\in \mathcal{Q}}\left[ \frac{1}{\beta }%
(1-\beta \tilde{\gamma})\log \int \left( \frac{dQ}{dP}\right) ^{1/(1-\beta 
\tilde{\gamma})}dP+\tilde{\gamma}\log E_{P}\left([e^{g}]^{1/\tilde{\gamma%
}}\right) \right] .
\end{equation*}%
From Lemma \ref{lem:convex_H} we see that the last term is convex in $\tilde{%
\gamma}$. Since $1-\beta \tilde{\gamma}$ is just an affine function of $%
\tilde{\gamma}$, it follows from Lemma \ref{lem:convex_H} again that the
first term is also convex in $\tilde{\gamma}$. This completes the proof.
\qed
\end{proof}

\subsection{A risk-sensitive scheduling control problem}\label{sec42}

We focus on one out of various RS control problems that are of interest in
the multi-class $G/G/1$ setting. In this setting each arrival requires a
single service. A recurring theme in the literature is how to schedule
service so as to minimize delay or queue length costs. The need to cover
general service time distributions has been recognized many times in earlier
work on this model. However, under RS cost, this question has only been
addressed in the Markovian setting. Our goal here is to show how the
perturbation bounds, specifically Theorem \ref{thm2}, can be used to yield
performance guarantees for the non-Markovian setting.

Let $N$ denote the number of classes, and for $i\in \{1,\ldots ,N\}$ let $%
X_{i}$, $A_{i}$ and $S_{i}$ denote the $i$th queue length process, arrival
process and potential service process. Then for each $i$, the balance
equation holds, namely 
\begin{equation*}
X_{i}(t)=X_{i}(0)+A_{i}(t)-S_{i}(U_{i}(t)),
\end{equation*}%
where $U_{i}(t)$ denotes the cumulative time devoted by the server to class $%
i$ by time $t$. In particular, $U_{i}$ are nondecreasing, Lipschitz
continuous with constant 1, and $\sum_{i}U_{i}(t)\leq t$ for all $t$. 
We regard $A$ and $S$ as {\it primitive processes}, and call $X$ and $U$
a {\it state process} and a {\it control process}, respectively, if $U$ is adapted to the
filtration $\calF_t=\sig\{A_i(s),X_i(s), s\le t,i\le N\}$.

It is assumed that $A_{i}$ and $S_{i}$ are mutually independent renewal
processes. In the $n$th system, $A_{i}$ and $S_{i}$ are replaced by $%
A_{i}^{n}=A_{i}(n\cdot )$ and $S_{i}^{n}=S_{i}(n\cdot )$, and the
corresponding control and queue length processes are denoted by $U^{n}$ and $%
X^{n}$, respectively. Normalized queue length is denoted by $\bar{X}%
_{i}^{n}=n^{-1}X_{i}^{n}$.

Fix $T>0$ and constants $c_{i}>0$. For $\beta >0$, denote
\begin{equation*}
J^{n}(U^{n};Q,\beta )=\frac{1}{n}\frac{1}{\beta }\log E_{Q}e^{\beta
\sum_{i}c_{i}X_{i}^{n}(T)}.
\end{equation*}%
There is redundancy in the definition with respect to $\beta $ and $c_{i}$.
We use the parameter $\beta $ to be consistent with Theorem \ref{thm2}, but
one could let $\beta =1$ without loss. \ 

Denote by $P$ the Markovian model, where $A_{i}$ and $S_{i}$ are Poisson
processes with parameters $\lambda _{i}$ and $\mu _{i}$, respectively.
Denoting 
\begin{equation*}
V^{n}(P,\beta )=\inf_{U^{n}}J^{n}(U^{n};P,\beta ),
\end{equation*}%
where the infimum ranges over control processes $U^n$,
the limit $V(P,\beta )=\lim_{n}V^{n}(P,\beta )$ was shown to exist and was
characterized in \cite{AGS1} as the viscosity solution of a HJB equation. In 
\cite{AGS2} it was proved that for zero initial conditions one has $V\leq
\beta ^{-1}WT$, where 
\begin{equation*}
W=W(\beta )=\min_{u\in \mathbf{u}}\sum_{i=1}^{N}(\hat{\lambda}_{i}-u_{i}\hat{%
\mu}_{i})^{+},
\end{equation*}%
$\mathbf{u}=\{u\in {\mathbb{R}}_{+}^{N}:\sum u_{i}\leq 1\}$, $\hat{\lambda}%
_{i}=\lambda _{i}(e^{\beta c_{i}}-1)$ and $\hat{\mu}_{i}=\mu
_{i}(1-e^{-\beta c_{i}})$. It was also shown in \cite{AGS2} that when $%
e^{-\beta c_{i}}<\lambda _{i}/\mu _{i}$ for all $i$ the bound is tight,
i.e., $V(\beta)=\beta ^{-1}WT$. In this case, it is asymptotically optimal to
prioritize according to the index $\mu _{i}(1-e^{-\beta c_{i}})$, regardless
of $T$, with larger values given priority.

Let $P$ be fixed as above, and consider a family $\mathcal{Q}$ defined via
part (ii) of Theorem \ref{th3}. That is, letting $h_{i,1}$ and $h_{i,2}$
stand for the hazard rates for $A_{i}$ and $S_{i}$, respectively, assume
that 
\begin{equation}\label{310}
a_{i,1}\leq \frac{h_{i,1}(\cdot )}{\lambda _{i}}\leq b_{i,1},\qquad
a_{i,2}\leq \frac{h_{i,2}(\cdot )}{\mu _{i}}\leq b_{i,2},
\end{equation}%
for some constants $0<a_{i,j}<b_{i,j}$. Denote by $Q_{T}^{n}$ and $P_{T}^{n}$
the law of $(A^{n},S^{n})|_{[0,T]}$ under $Q$ and $P$. Then by \eqref{201},
for all $Q\in \mathcal{Q}$, 
\begin{equation*}
R_{\alpha }(Q_{T}^{n}\Vert P_{T}^{n})\leq nTf_{0}(\alpha ),\qquad \text{where%
}\qquad f_{0}(\alpha )=\sum_{i}\left[ k_{\alpha }(a_{i,1})\vee k_{\alpha
}(b_{i,1})\right] \lambda _{i}+\sum_{i}\left[ k_{\alpha }(a_{i,2})\vee
k_{\alpha }(b_{i,2})\right] \mu _{i}
\end{equation*}%
[we recall $k_{\alpha }(x)=[x^{\alpha }-x\alpha +\alpha -1]/\alpha (\alpha
-1)$ introduced in \eqref{094}]. Thus Theorem \ref{thm2} may be applied with $f(\alpha
)=nTf_{0}(\alpha )$. Denoting 
\begin{equation*}
V^{n}(\mathcal{Q},\beta )=\inf_{U^{n}}\sup_{Q\in \mathcal{Q}%
}J^{n}(U^{n};Q,\beta ),
\end{equation*}%
we have by Theorem \ref{thm2} that, for all $n$, 
\begin{equation}
V^{n}(\mathcal{Q},\beta )\leq \inf_{U_{n}}\inf_{\gamma >\beta }\left[ \frac{%
Tf_{0}(\frac{\gamma }{\gamma -\beta })}{\gamma -\beta }+J^{n}(U^{n};P,\gamma
)\right] =\inf_{\gamma >\beta }\left[ \frac{Tf_{0}(\frac{\gamma }{\gamma
-\beta })}{\gamma -\beta }+V^{n}(P,\gamma )\right] .  \label{308}
\end{equation}

As mentioned above, for each $\gamma $, $\limsup_{n}V^{n}(P,\gamma )\leq
\gamma ^{-1}W(\gamma )T$ (according to Theorem 2.1 of \cite{AGS2}). Hence
we obtain the following upper bound.
\begin{theorem}\label{th08}
For each $\beta$, one has
\begin{equation}
\limsup_{n}V^{n}(\mathcal{Q},\beta )\leq
B(\calQ,\beta)\doteq \inf_{\gamma >\beta }\left[ \frac{%
f_{0}(\frac{\gamma }{\gamma -\beta })}{\gamma -\beta }+\frac{W(\gamma )}{%
\gamma }\right] T.  \label{309}
\end{equation}%
\end{theorem}
\noi{\bf Proof.}
To deduce \eqref{309} from \eqref{308}, let ${\varepsilon }>0$ be
given and let $\gamma _{0}$ be ${\varepsilon }$-optimal for the RHS of %
\eqref{309}. Then use \eqref{308} with the infimum over $\gamma $ replaced
by the substitution $\gamma =\gamma _{0}$. Taking the limit and using the
aforementioned limit result for $V^{n}(P,\gamma _{0})$, then sending ${%
\varepsilon }\rightarrow 0$, gives \eqref{309}.
\qed

\skp

We now show that using the last result one can identify a policy
for which the robust bound \eqref{309} is valid.
\begin{remark}
The robust bound derived in Theorem \ref{th08} can in fact be achieved by an index policy.
Given $\beta$ and $\eps>0$, let
$\gamma ^{\ast }$ be an $\eps$-minimizer of the RHS of \eqref{309}.
Let us show that
prioritizing queues according to the index $\mu _{i}(1-e^{-\gamma ^{\ast }c_{i}})$
(rather than $\mu _{i}(1-e^{-\beta c_{i}})$) guarantees the uniform bound
\[
\limsup_{n}\sup_{Q\in\calQ}J^n(U^{*,n};Q,\beta)\le B(\calQ,\beta)+\eps,
\]
where $U^{*,n}$ denotes the control corresponding to the above mentioned index.
To see this, note that by Theorem~\ref{thm2}, when $U^{*,n}$ is implemented
one has for all $n$
\[
J^{n}(U^{*,n};\mathcal{Q},\beta )
\leq
\inf_{\gamma >\beta }\left[
\frac{Tf_{0}(\frac{\gamma }{\gamma -\beta })}{\gamma -\beta }+J^{n}(U^{*,n};P,\gamma)
\right]
\le\left[ \frac{Tf_{0}(\frac{\gamma^* }{\gamma^* -\beta })}{\gamma^* -\beta }+J^{n}(U^{*,n};P,\gamma^*)\right],
\]
and since by \cite{AGS2}, under the fixed priority $U^{*,n}$ one has that
$J^n(U^{*,n};P,\gamma^*)$ converges to $(\gamma^*)^{-1}TW(\gamma^*)$, the claim follows.
\end{remark}

Define $\ell (x)=x\log x-x+1$ for $x\geq 0$. When $\beta$ is small a
natural assumption to make on the $a_{i,k}$ and $b_{i,k}$, consistent with
the fact that the R\'{e}nyi rate $k_{\alpha }(x)$ becomes the relative entropy
rate $\ell (x)$ as $\alpha \downarrow 1$, is that $a_{i,k}<1<b_{i,k}$ and $%
\ell (a_{i,k})=\ell (b_{i,k})$ so long as $\ell (b_{i,k})\leq 1$, and $%
a_{i,k}=0$ if $\ell (b_{i,k})>1$. In this case one can show that the bounds
are also tight in a precise sense, which is that there exists a model in $%
\mathcal{Q}$ such that the two sides differ by no more than error term that
vanishes as $\beta \downarrow 0$ and which can be calculated.

As pointed out above, the robust RS control policy thus obtained prioritizes according
to an index that is distinct from that used for the reference model.
This illustrates an important aspect of the general approach of using
Theorem \ref{thm2} and \eqref{306}, namely that
there is more to this approach than directly applying the \renyi bounds
to the state process obtained under the optimal RS control
for the reference model $P$.
Indeed, the latter approach would give rise to a control for $(\calQ,\beta)$
that agrees with that for $(P,\beta)$. Instead, the minimization problem \eqref{306}
allows for the control (and consequently the state process)
to differ from the one that is optimal for $(P,\beta)$ by allowing
freedom in choosing the sensitivity parameter $\gamma$.
Thus $\gamma$ is selected to best fit
the family $\calQ$, which may indeed result in a control policy
that is not optimal for the `reference problem' $(P,\beta)$.

\begin{example}
We evaluate the bound $B(\calQ,\beta)$ of \eqref{309} numerically.
We consider an example with 5 classes, with data
$\la=(1, 1.5, 1.8, 2, 2)$ and $\mu=(8, 10, 12, 9, 14)$.
The overall traffic intensity $\rho=\sum_i\frac{\la_i}{\mu_i}$ is $\rho=0.790$.
The relative costs $c_i$ are taken to be $c=(0.3, 0.2, 0.2, 0.1, 0.2)$,
and the time horizon $T=1$.

First the reference model is considered.
When $\calQ$ is a singleton consisting
of the model $P$, the bound is
$B(\calQ,\beta)=\beta^{-1}W(\beta)=V(\beta)$.
This function is shown in blue is Figure \ref{fig2} (left),
for $\beta$ in the range $[0,15]$.

Consider the family determined by \eqref{310}, with
$a_{i,1}=a_{i,2}=1-\del$ and $b_{i,1}=b_{i,2}=1+\del$ for all $i$,
for $\del=0.65$. Recall that this corresponds to a family of models
driven by renewal processes,
where the interarrival and service time distributions have hazard rates
that deviate from the respective Poisson rates of the reference model by
at most 65\%.
Moreover, according to Theorem \ref{th3}(ii), this may also stand for
a family of models where the driving processes are Cox, for which
the stochastic intensities deviate from those of the reference model
by at most 65\%.
This family is denoted by $\calQ_2$ (for it corresponds to part (ii) of Theorem \ref{th3}).
In Figure \ref{fig2} (left) the bound $B(\calQ_2,\beta)$ the RS cost for this family
is shown in solid black line.
A dotted black line shows the bound $B(\calQ_2,\beta)$
where now the parameter $\del$ is taken as $\del=0.15$.

Next, consider the family
of models, denoted by $\calQ_3$, for which the driving processes are
as in Theorem \ref{th3}(iii).
These are Cox processes
which, in addition to bounds on the deviation from
the reference Poisson rates, the stochastic intensities
satisfy a long run average constraint.
For example, the potential service process for class 1
has stochastic intensity that deviates from $\mu_1$ by at most
$\del$, and in addition is constrained to have a long run average
equal to $\mu_1$.
In Figure \ref{fig2} (left), the bound $B(\calQ_3,\beta)$ is shown
in solid red line and in dotted red line for $\del=0.65$
and $\del=0.15$, respectively.
As expected, the bounds for $\calQ_3$ are smaller than for $\calQ_2$,
and they are smaller for $\del=0.15$ than they
are for $\del=0.65$.

Finally, all five graphs are repeated in Figure \ref{fig2} (right)
with a different $\la$, namely
$\la=(0.5, 0.75, 0.9, 1, 1)$ (leaving the remaining parameters
unchanged) in which case $\rho=0.395$. The queueing
system is more stable in this case, and the performance
guarantees, as measures by the RS cost bounds, are smaller
as expected.

\end{example}

\begin{figure}
\center{
\includegraphics[width=0.48\textwidth,height=0.45\textwidth]{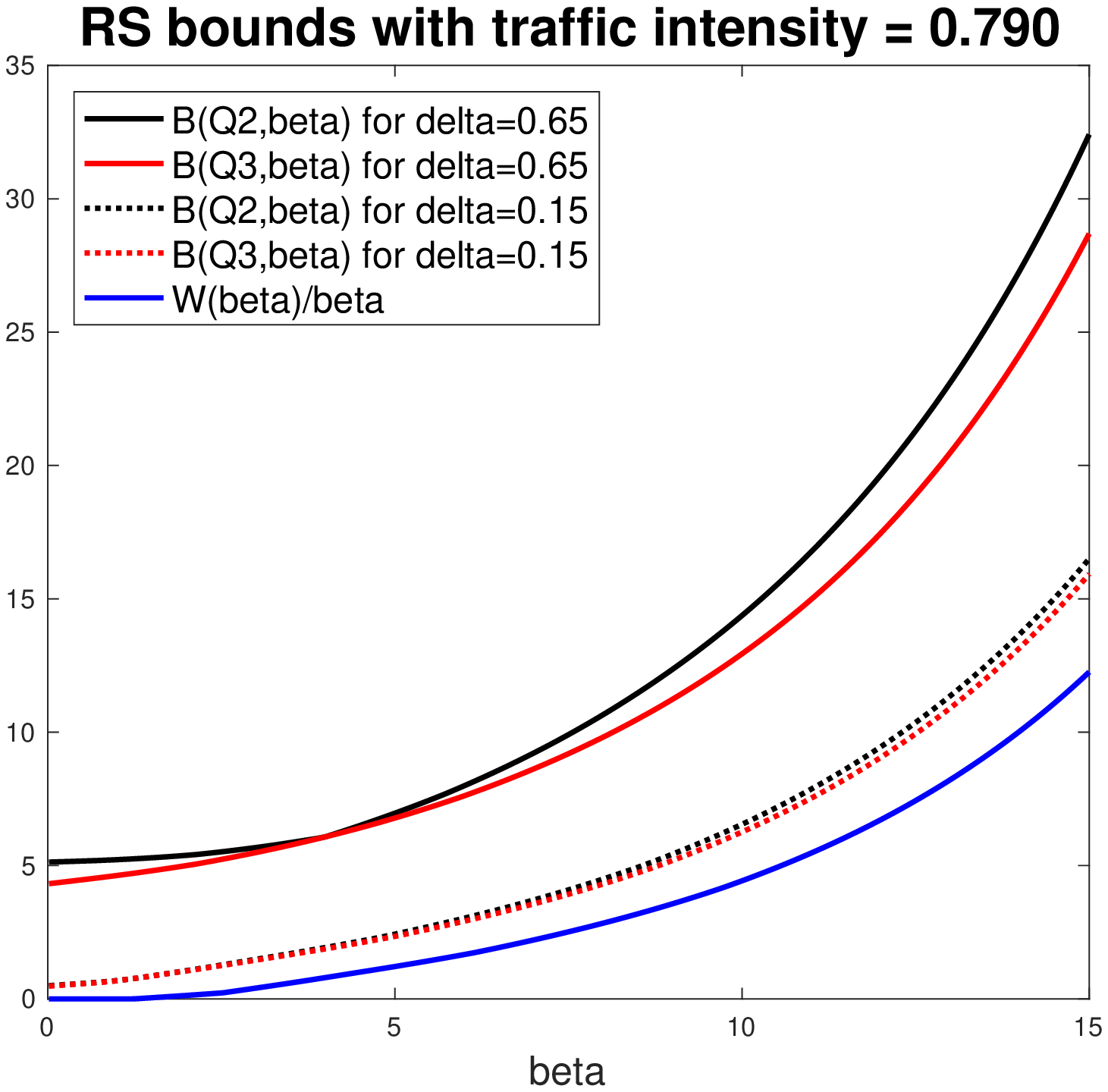}
\includegraphics[width=0.48\textwidth,height=0.45\textwidth]{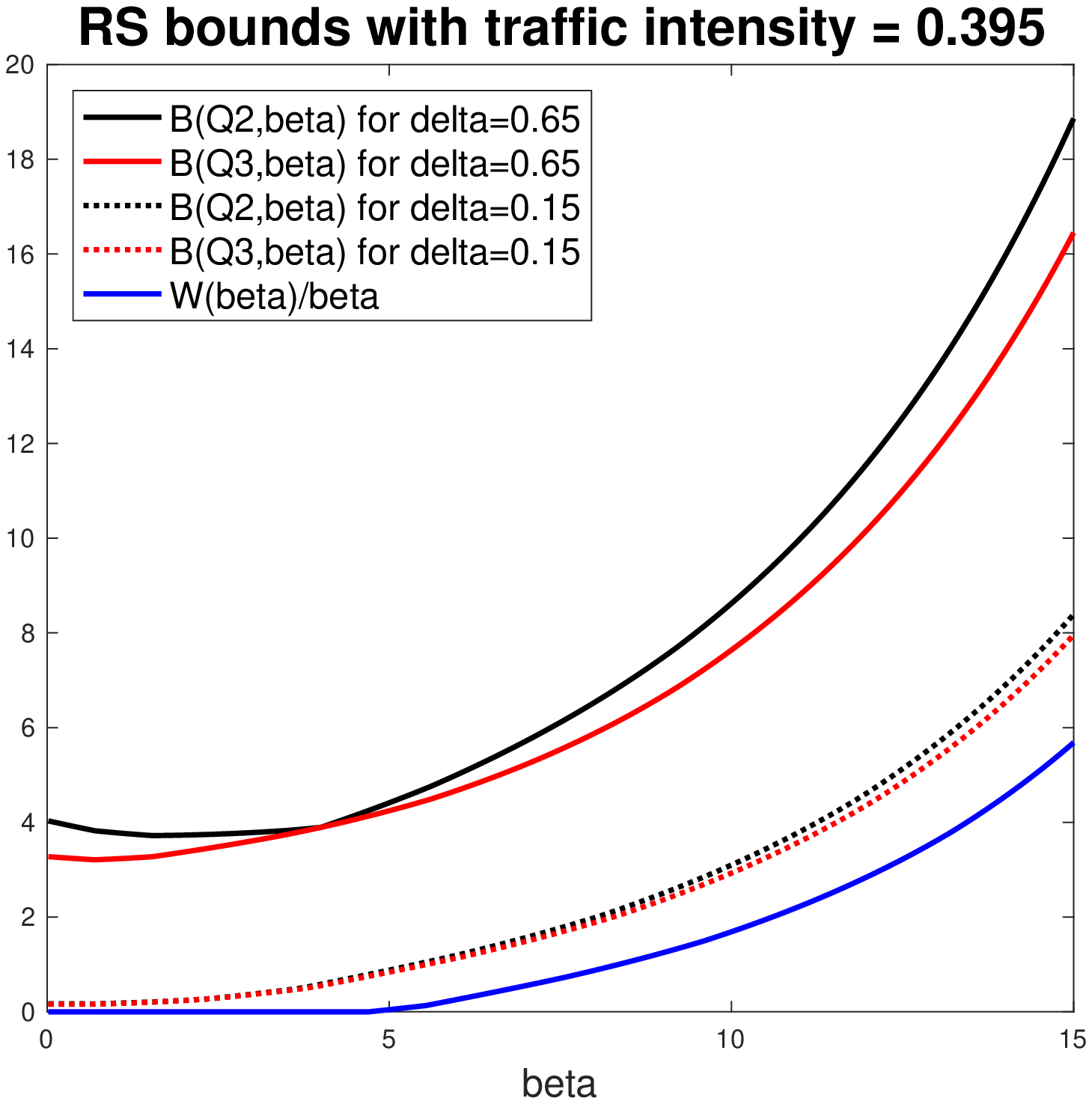}
}
\caption{Robust risk-sensitive bounds for the scheduling problem}
\label{fig2}
\end{figure}

\section{Robust LD estimates for queueing models with reneging}
\beginsec
\label{sec5}

In this section we study a multi-server queue with reneging, under a scaling
where the number of servers $n$ and the arrival process grow proportionally. This scaling
has been referred to as a \textit{many-server} scaling,
studied for the first time in \cite{halfin}, for CLT asymptotics in the case of exponential servers,
and then in the context of general service times in \cite{kasram11},
\cite{kasram13}, \cite{reed09}, \cite{kanram10} (for LLN and CLT asymptotics).
For models that accommodate reneging, it is natural to define performance in terms
of the reneging count. This was addressed recently in
\cite{ABDW1}, where the large time, large $n$ asymptotics of
the probability of atypically large reneging count was identified precisely.
These results were concerned with the Markovian $M/M/1+M$ and $M/M/n+M$
models. Whereas the results of \cite{ABDW1} identify exact LD asymptotics for one particular model,
our interest here is in the spirit of robust bounds,
in estimates that are uniform within families of models, that are moreover non-Markovian.
The results from \cite{ABDW1} will serve us as reference for these uniform bounds.

Treating general service time distributions via a Markovian reference model
relies, according to our approach, on \renyi divergence estimates of
the underlying primitives, which in this case are given by the potential service processes
for each server. This is precisely where our results from Section \ref{sec3} on divergence
of various counting processes w.r.t.\ Poisson become useful.
A similar remark holds for other primitives of the model, namely arrival and patience times.
Specifying server characteristics by means of a counting process that lies
in a given \renyi radius about some nominal Poisson gives room for modelling
servers as different from one another. In fact, in this framework
there is no benefit to requiring that servers be statistically identical.
This gives rise to a set of models much more rich than $G/G/n+G$, that
accommodates (a) heterogeneous servers, and (b) time varying processing capacities.

Whereas item (a) above
allows for distinct probabilistic characteristics for each server,
our approach is to express the degree of uncertainty
(w.r.t.\ service times distributions) on equal terms for all servers.
%This leads to an important distinction:
This should not be confused with models such as
%Thus servers are treated in the same way in terms of model uncertainty,
$G/G/n$ or $G/G/n+G$ where all servers operate under the same distribution.
This modelling approach is perhaps more satisfactory than models like $G/G/n$
in situations where
there is no information that distinguishes between servers but at the same time
there is no reason to believe that all are identical.
An analogous remark is valid for modelling patience of different customers.

\subsection{Model and performance measure}

\subsubsection{Model equations}

Customers arrive at the system with service requirement that can be handled
by any one of $n$ parallel servers. They are queued if no servers are available
upon arrival, and renege if they are still in the queue at the time their patience expires.
The priority within the queue is according to FIFO. Determining which available server
takes the next customer is according to a fixed ordering of the servers.

Because, on the one hand, servers have different characteristics and, on the other hand, customer
reneging depends on their state (specifically, whether they are in the queue and for how long),
the model equations must account for
the state of each server as well as the state of each customer.
Hence our system of equations
will be based on a balance equation for each server and one for each
(of the infinitely many) customers.

The model equations are therefore somewhat complicated. However, because our approach is based
on the existence of a mapping from primitive processes to the full state of the system,
it is necessary to write down these equations so that a concrete mapping is indeed well defined.
Measure valued processes are often used for encoding the dynamics, however
it seems less complicated in the current context to write balance equations, as we will.
Also, we are careful to write the equations without relying on an assumption
that the underlying discrete events occur one at a time; that is, they allow for
the possibility of
simultaneous arrival and departure, simultaneous departures at different servers, etc.
This assures that the mapping is defined on the full path space of the primitive processes.

The customers are indexed by $\N$, and a marked point process $\sum_{i\in\N}\del_{(T_i,P_i)}$,
with sample paths in $\calD([0,\iy):\calM_F(\R_+^2))$
encodes their time of arrival $T_i$ and their patience time $P_i$.
With a slight abuse of notation, in what follows we refer to $\calA=(T_i,P_i)$
as the marked point process. It is assumed that
$0\le T_1\le T_2\le\cdots$ and $P_i>0$ for all $i$.
Those customers $i$ with $T_i=0$ are initially in the system.
The $n$ servers are indexed by $[n]=\{1,\ldots,n\}$,
and a counting process $S_j$ is associated with each server $j\in[n]$,
representing its potential service process. That is, $S_j(t)$ customers depart server
$j$ by the time this server has worked for $t$ units of time.

We start with a balance equation for each customer. For $i\in\N$, we have
\[
Q_i(t)=A_i(t)-K^{\rm cust}_i(t)-R_i(t).
\]
Here, the four processes $Q_i$, $A_i$, $K^{\rm cust}_i$ and $R_i$
are $\{0,1\}$ valued, representing queueing, arrival, routing
and reneging, respectively, associated with customer $i$.
Thus $Q_i$ (resp., $A_i$, $K^{\rm cust}_i$, $R_i$) takes the value 1 at time $t$
if customer $i$ in the queue at that time (resp., has arrived prior to or at $t$, has been routed to service
prior to or at $t$, has reneged prior to or at $t$). In particular, we have $A_i(t)=1_{\{t\ge T_i\}}$.

Next, a balance equation holds for each server $j\in[n]$, in the form
\[
B_j(t)=B_j(0)+K^{\rm serv}_j(t)-D_j(t), \qquad D_j(t)=S_j\Big(\int_0^tB_j(s)ds\Big).
\]
Here, $B_j$, $K^{\rm serv}_j$ and $D_j$ are busyness, routing and departure processes
associated with server $j$, taking values in $\{0,1\}$, $\Z_+$ and $\Z_+$, resp.
Namely, $B_j$ takes the value $1$ at $t$ if the server is busy,
and $K^{\rm serv}_j$ and $D_j$ are counting processes for the number customers routed
to and, resp., departing from server $j$.

The initial conditions are assumed to match and to satisfy a work conservation condition.
Namely, the number of customers initially in the system, $X(0)=\max\{i:T_i=0\}$,
and the number of servers initially busy, $B(0)=\sum_jB_j(0)$, satisfy
$B(0)=X(0)\w n$.

Next we describe how the routing processes are determined so as to keep the aforementioned priority
rules.
To this end, we denote by
\[
AV^{\rm cust}(t)=\{i\in\N: \text{ either } Q_i(t-)=1 \text{ or }\Del A_i(t)=1\}
\]
the set of customers available for routing at time $t$ and by
\[
AV^{\rm serv}(t)=\{j\in[n]: \text{ either } \Del D_j(t)=1\text{ or } B_j(t-)=0\}
\]
the set of servers available to serve at this time,  where for a real valued c\`{a}dl\`{a}g function $f$ on $[0,\infty)$, $\Delta f(s) = f(s)- f(s-)$.
The number of customers to be routed at time $t$ is given by
\[
\hat K(t)=\# AV^{\rm cust}(t)\w\# AV^{\rm serv}(t).
\]
In terms of $\hat K(t)$, one determines which customers $i$ are routed to service,
and which servers $j$ admit new customers at time $t$, according to
\[
\hat K^{\rm cust}_i(t)=\begin{cases}1 & \text{if } i\in AV^{\rm cust}(t),\ \#\{i'\in AV^{\rm cust}(t):i'\le i\}\le\hat K(t),
\\
0 & \text{otherwise},
\end{cases}
\]
\[
\hat K^{\rm serv}_j(t)=\begin{cases}1 & \text{if } j\in AV^{\rm serv}(t),\ \#\{j'\in AV^{\rm serv}(t):j'\le j\}\le\hat K(t),
\\
0 & \text{otherwise}.
\end{cases}
\]
The corresponding counting processes are given by
\[
K^{\rm cust}_i(t)=\sum_{s\le t}\hat K^{\rm cust}_i(s),
\]
\[
K^{\rm serv}_j(t)=\sum_{s\le t}\hat K^{\rm serv}_j(s).
\]

To determine $R_i$, note that reneging occurs at time $T_i+P_i$, but only on the event that
the customer is in the queue at that time.
Thus
\[
R_i(t)=\begin{cases}
1 & \text{if } t\ge T_i+P_i \text{ and } K^{\rm cust}_i(T_i+P_i)=0,
\\
0 & \text{otherwise}.
\end{cases}
\]
According to this definition, if routing of a customer to service and reneging potentially occur
at the same time, priority is given to routing.

Finally, the total queue length, number of busy servers, arrival count,
departure count, reneging count
and routing count are given, resp., by
\[
Q=\sum_iQ_i,\qquad B=\sum_jB_j, \qquad
A=\sum_iA_i,\qquad D=\sum_jD_j,
\]
\[
R=\sum_iR_i, \qquad K=\sum_iK^{\rm cust}_i=\sum_jK^{\rm serv}_j.
\]
The state of the system is the process
$\Sig=\{A_i,Q_i,K^{\rm cust}_i,R_i,B_j,K^{\rm serv}_j,D_j,Q,B,A,D,R,K\}$.

The primitive processes $\calA=(T_i,P_i)$ and $S_j$ determine the state of the system.
Proving this amounts to showing that there exists a unique solution to the set of all equations
that appear in this subsection; we skip the details of the elementary proof
of this fact. An additional important fact that we state without proof
is a causality property, namely that for any $t$, $\{\Sig(s):s\in[0,t]\}$ is measurable on
the sigma field corresponding to the primitive data up to time $t$,
namely $\s\{(T_i,P_i)1_{\{T_i\le t\}}, S_j(s), s\le t, j\in[n]\}$.

It is assumed throughout that the potential service processes $S_j$
are mutually independent, and that, moreover,
the initial data $(\{B_j(0)\}_{j\in\N}, X(0))$, the service primitive $\{S_j\}_{j\in\N}$
and the customer primitive $\calA$ are mutually independent,
for each model $\Q$ in the family of models $\calQ$ to be considered.

\subsubsection{LD scaling and performance measure}

We now consider a sequence of models indexed by $n\in\N$.
It is convenient to assume, as we will, that the primitives $S_j$ are given for all $j\in\N$,
and that for the $n$th system, one takes $S^n_j=S_j$, $j\in[n]$.
Similarly, for the arrival process, it is convenient to start with a single sequence
$\calA=(T_i,P_i)$ and obtain the arrival process for the $n$th system, $\calA^n=(T_i^n,P_i^n)$,
via $T^n_i=n^{-1}T_i$ and $P^n_i=P_i$.
The transformation of arrival times reflects acceleration of arrivals,
performed in order to keep a constant traffic intensity as $n$ increases
by balancing the increase of processing capacity due to the growing
number of servers.
The patience times however are not accelerated. This is in agreement with
literature on many server scaling at LLN and CLT regimes,
such as \cite{atar2004scheduling}, \cite{kanram10}, \cite{atakasshi}.
The superscript $n$ is attached to all processes involved in the $n$th system.

Our main interest is in the LD behavior of the reneging count $R^n$.
In the special case of Markovian model, the large time average rate of overall reneging
can be obtained by simple LLN considerations. That is,
assume that for some $\la,\mu,\theta>0$,
the rate of arrivals is given by $\la n$, the total service rate by $\mu n$,
and the per-customer reneging rate by $\theta$.
Consider an overloaded system, $\la>\mu$.
Then the reneging stabilizes the system at
an equilibrium around $xn$ for which $\la = \mu+\theta x$.
Hence the long time average reneging rate is given by $\gamma_0=\theta x
=\theta\frac{\la-\mu}{\theta}=\la-\mu$, and thus for $\gamma>\gamma_0$,
the event that the long time average reneging rate exceeds $\gamma$ is rare.

For a general model $\Q$ and an arbitrary $\gamma>0$, define the decay rate
\[
\chi(\Q,\gamma)=\limsup_{t\to\iy}\limsup_{n\to\iy}\frac{1}{tn}
\log \Q\Big(\frac{R^n(t)}{t n}>\gamma\Big),
\]
and for a collection of models $\calQ$ let
\[
\chi(\calQ,\gamma)=
\limsup_{t\to\iy}\limsup_{n\to\iy}\frac{1}{tn}
\log \sup_{\Q\in\calQ}\Q\Big(\frac{R^n(t)}{t n}>\gamma\Big).
\]

Bounds on $\chi(\calQ,\gamma)$ will be based on known
bounds on $\chi(\PP,\gamma)$, where $\PP$ stands for
the aforementioned Markovian model (that is, $M/M/n+M$), and $\gamma>\gamma_0$.
%Recall the parameters $\la,\mu,\theta>0$ standing for the arrival rate,
%per-server service rate, and per customer reneging rate are given by $\la n$, $\mu$ and $\theta$.

\begin{theorem}{\bf\cite{ABDW1}}
\label{th01}
Assume $\la\ge\mu$. Let
\(
C(\gamma)=\la(1-z^{-1})+\mu(1-z)-\gamma\log z
\), where
\[
z=z(\gamma)=\frac{\sqrt{\gamma^2+4\mu\la}-\gamma}{2\mu}.
\]
Then $\chi(\PP,\gam)=-C(\gam)$, for $\gamma\ge\gamma_0$.
\end{theorem}

\subsection{Robust bounds}

\subsubsection{Robust bounds in general form}

For a collection of models $\calQ$,
the marked point process
$\calA^n=(T^n_i,P^n_i)$, with $\calA^n_t=\{(T^n_i,P^n_i):T^n_i\le t\}$,
which encodes arrival and patience processes, is assumed to satisfy the RDR bound
\begin{equation}\label{90}
\limsup_{t\to\iy}\limsup_{n\to\iy}
\frac{1}{nt}\sup_{\Q\in\calQ}R_\al(\Q\circ\calA^n|_{[0,t]}^{-1}
\|\PP\circ\calA^n|_{[0,t]}^{-1})\le r^{(1)}_\al,
\end{equation}
where $r_\al^{(1)}$ is an $\al$-dependent constant.
The probabilistic characteristics of the servers are encoded in
the service processes $S_j$, that are taken to satisfy a similar bound, uniform in $j$,
\begin{equation}\label{91}
\limsup_{t\to\iy}\frac{1}{t}\sup_{\Q\in\calQ}\sup_{j\in\N}R_\al(\Q\circ S_j|_{[0,t]}^{-1}\|\PP\circ S_j|_{[0,t]}^{-1})\le r^{(2)}_\al.
\end{equation}

\begin{theorem}\label{th4}
Assume \eqref{90} and \eqref{91}. Then we have, for every $\gam\ge\gam_0$,
the estimate
\begin{equation*}
%\label{200}
\chi(\calQ,\gamma)\le B(\calQ,\gamma)\doteq
\inf_{\al>1}\Big[
-\frac{\al-1}{\al}C(\gamma)+(\al-1)(r_\al^{(1)}+r_\al^{(2)})\Big].
\end{equation*}
\end{theorem}

\noi{\bf Proof.}
Clearly, the event $(tn)^{-1}R^n(t)>\gam$ is measurable on $\sig\{\Sig^n(s):s\in[0,t]\}$,
hence on the sigma field corresponding to the primitives, $\sig\{\calA^n_s,S_j(s):s\in[0,t],j\in[n]\}$.
Hence by \eqref{09},
\begin{align*}
\sup_{\Q\in\calQ}\frac{1}{tn}\log\Q\Big(\frac{R^n(t)}{tn}>\gam\Big)
&\le \frac{\al-1}{\al}\frac{1}{tn}\log\PP\Big(\frac{R^n(t)}{tn}>\gam\Big)\\
&\quad+(\al-1)\frac{1}{tn}\sup_{\Q\in\calQ}
\Big[R_\al(\Q\circ\calA^n|_{[0,t]}^{-1}\|\PP\circ\calA^n|_{[0,t]}^{-1})
+\sum_{j=1}^nR_\al(\Q\circ S_j|_{[0,t]}^{-1}\|\PP\circ S_j|_{[0,t]}^{-1})\Big],
\end{align*}
using the assumed mutual independence of the different service processes $S_j$,
as well as their independence from $\calA^n$.
Taking the limit superior in $n$, then in $t$, and using Theorem \ref{th01}
and assumptions \eqref{90} and \eqref{91}, gives
\[
\chi(\calQ,\gam)\le-\frac{\al-1}{\al}C(\gam)+(\al-1)(r_\al^{(1)}+r_\al^{(2)}).
\]
The result follows on optimizing over $\al$.
\qed

\subsubsection{Examples}

First we provide examples where uncertainty classes are in the spirit if Theorem \ref{th3}.
In the reference model $\PP$, arrivals are Poisson$(n\la)$, patience times are exponential$(\theta)$,
and individual service rates are $\mu=1$.

\begin{example}\label{ex-3}
Consider the family of models, $\calQ_2$, corresponding to
the setting of Theorem \ref{th3}(ii), where all potential service processes
are Cox processes with stochastic intensities
$a\le \la_j(\cdot)\le b$, for some constants $0<a\le 1\le b<\iy$
and all $j\in\N$. Alternatively, (again, see Theorem \ref{th3}(ii))
the service processes are renewals
with service time distribution for which the hazard rate satisfies $a\le h_j(\cdot)\le b$
for each server $j\in\N$. As already mentioned, the potential service
processes are assumed to be mutually independent, but they are not assumed
to be identically distributed. However, the distributions associated with all servers
are assumed to lie within the same uncertainty class, namely the one determined
by the bounds $a$ and $b$.
For arrival and patience distributions, recall that $\calA^n=(T^n_i,P^n_i)$
are taken as rescaled versions of $\calA=(T_i,P_i)$. Our assumptions
are that $T_i$ is
a Cox process with stochastic intensity $a_{\rm arr}\le\frac{\la(\cdot)}{\la}\le b_{\rm arr}$,
and the distributions of the patience times $P_i$ have distributions $\psi_i(z)\varsigma(dz)$ satisfying
the bound $a_{\rm pat}\le\psi_i(\cdot)\le b_{\rm pat}$, where $\varsigma$ is the
distribution of an exponential($\theta$) RV $\varsigma(dz)=\theta e^{-\theta z}dz$.

%To use Theorem \ref{th3}(ii), note first that $a_{\rm arr}a_{\rm pat}\le \frac{\la(s)\psi(z)}{\la} \le b_{\rm arr}b_{\rm pat}$.
For the potential service processes, consider Theorem \ref{th3}(ii) in the special case of no marks.
Then, for each server $j$,
\[
\sup_{\Q\in\calQ_2}R_\al(\Q\circ S_j|_{[0,t]}^{-1}\|\PP\circ S_j|_{[0,t]}^{-1})
\le
\sup_{\Q\in\calQ_2^*}R_\al(\Q\circ N|_{[0,t]}^{-1}\|\PP\circ N|_{[0,t]}^{-1}),
\]
where $\calQ_2^*$ denotes the class from Theorem \ref{th3}(ii), of all probability
measures that make $N$ a Cox process with stochastic intensity
$a\le\la(\cdot)\le b$ (taking $\la_0=1$ in \eqref{093}, in line with the assumption that $\mu=1$ under the reference model).
Taking the supremum over $j$ we obtain from Theorem \ref{th3}(ii) that
\eqref{91} holds with $r_\al^{(2)}=k_\al(a)\vee k_\al(b)$.

Next, to use the same theorem for the arrival and patience distributions,
note first that
$\hat a\le \frac{\la(s)\psi(z)}{\la} \le \hat b$,
where $\hat a=a_{\rm arr}a_{\rm pat}$ and $\hat b=b_{\rm arr}b_{\rm pat}$.
Consequently, we obtain \eqref{90} with $r_\al^{(1)}=k_\al(\hat a)\vee k_\al(\hat b)$.

Appealing to Theorem \eqref{th4}, we obtain the estimate
\[
\chi(\calQ_2,\gam)\le
\inf_{\al>1}\Big[
-\frac{\al-1}{\al}C(\gamma)+(\al-1)
(k_\al(\hat a)\vee k_\al(\hat b))\la
+(\al-1)(k_\al(a)\vee k_\al(b))\Big].
\]

As a second and third uncertainty classes, denoted by $\calQ_3$ and $\calQ_4$,
we take families of measures corresponding to Theorem \ref{th3}(iii) and \ref{th3}(iv),
respectively. In both $\calQ_3$ and $\calQ_4$,
the arrival and patience are taken as in $\calQ_2$.
As for service time distributions, in $\calQ_3$ consider Cox processes
for which the stochastic intensity satisfies the long time average constraint
\eqref{096} (with $\la_0=1$) and the constraint $a\le \la_j(\cdot)\le b$.
In $\calQ_4$, the stochastic intensity satisfies
\eqref{096} and the constraint \eqref{095} for some $\al=\al_0$ and $u$.
The bounds obtained in these cases are
\[
\chi(\calQ_3,\gam)\le
\inf_{\al>1}\Big[
-\frac{\al-1}{\al}C(\gamma)+(\al-1)
(k_\al(\hat a)\vee k_\al(\hat b))\la
+(\al-1)(pk_\al(a)+qk_\al(b))\Big],
\]
$p=\frac{b-1}{b-a}$, $q=\frac{1-a}{b-a}$, and
\[
\chi(\calQ_4,\gam)\le
\inf_{\al\in(1,\al_0)}\Big[
-\frac{\al-1}{\al}C(\gamma)+(\al-1)
(k_\al(\hat a)\vee k_\al(\hat b))\la
+\frac{1}{\al}
\Big[\Big(\bar\al_0u+1\Big)^{\frac{\al-1}{\al_0-1}}-1\Big]
\Big],
\]
where $\bar\al_0=\al_0(\al_0-1)$.

For a numerical example we take the following numerical values.
Since the reference service rates are normalized to 1,
the reference system will be overloaded in $\la>1$. We take $\la=2$.
The corresponding LLN reneging rate is $\gam_0=1$.
As bounds on intensities we take $a=\hat a=1-\del$ and $b=\hat b=1+\del$, where $\del=0.3$.

Figure \ref{fig3} (left) gives graphs of $B(\calQ_2,\gamma)$ and $B_3(\calQ_3,\gamma)$
corresponding to the families $\calQ_2$ and $\calQ_3$,
as well as $\chi(\PP,\gamma)=-C(\gamma)$ for reference (the exact decay rate under $\PP$).
In addition to the families $\calQ_2$ and $\calQ_3$, we consider families $\calQ_2'$ and $\calQ_3'$
defined analogously to $\calQ_2$ and $\calQ_3$, respectively, but where uncertainty is associated with
the service processes only, hence $r^{(1)}_\al$ is taken to be $0$. The corresponding bounds
$B(\calQ_2',\gamma)$ and $B(\calQ_3',\gamma)$ are also shown in Figure \ref{fig3} (left).
\end{example}

Whereas the above example is based on RDR bounds
for families of processes (Theorem \ref{th3}), the following is based, in addition, on
our RDR bounds for specific renewal distributions (Theorem \ref{thm1}).

\begin{example}\label{ex-4}
We consider a family, denoted $\calQ_\Gam$,
where all servers operate according to Gamma distributions.
More precisely, service time distribution for server $j$ is $\Gam(k_j,\rho_j)$,
$k_j\ge1$, $\rho_j>1$, and a subset $F\subset[1,\iy)\times(1,\iy)$ is given
for which $(k_j,\rho_j)\in F$ for all $j$.
The assumptions on the arrival and patience process, $\calA$, are as in Example \ref{ex-3}.
Then by the bound on RDR for the Gamma distribution stated in
Example \ref{ex:32}, the bound \eqref{91} is valid with
\[
r_\al^{(2)}=r_\al^{(2)}(F)\doteq\sup_{(k,\rho)\in F}r_\al^{(2)}(k,\rho),
\]
where we denote
\[
r_\al^{(2)}(k,\rho)=
\Big( \frac{\Gam(1+\alpha(k-1))}{(\Gam(k))^\alpha} \rho^{\alpha k} \Big)
^{\frac{1}{1+\alpha(k-1)}} -\alpha(\rho-1) - 1.
\]
As a consequence,
\[
\chi(\calQ_\Gam,\gamma)
\le\inf_{\al>1}\Big[-\frac{\al-1}{\al}C(\gamma)
+(\al-1)(k_\al(\hat a)\vee k_\al(\hat b))\la
+(\al-1)r_\al^{(2)}(F)\Big].
\]

Figure \ref{fig3} (right) gives graphs of $B(\calQ_\Gam)$
for $\calQ_\Gam$ for two parameter ranges, namely $(k,\rho)\in[1,1.1]\times[1,1.1]$ and
$(k,\rho)\in[1,1.5]\times[1,1.5]$.
The parameters $\la$, $\hat a$, $\hat b$ are taken to be $\la=2$
$\hat a=1-\del$, $\hat b=1+\del$, where $\del=0.3$.

Finally, we also consider $\calQ_\Gamma'$
defined as $\calQ_\Gamma$ but where uncertainty is associated with
the service processes only ($r^{(1)}_\al=0$).
The corresponding bounds are also shown in Figure \ref{fig3} (right),
with the same ranges of parameters $(k,\rho)$.

Once again, $\chi(\PP,\gamma)$ is also plotted for reference.
\end{example}

\begin{figure}
\center{
\includegraphics[width=0.48\textwidth,height=0.45\textwidth]{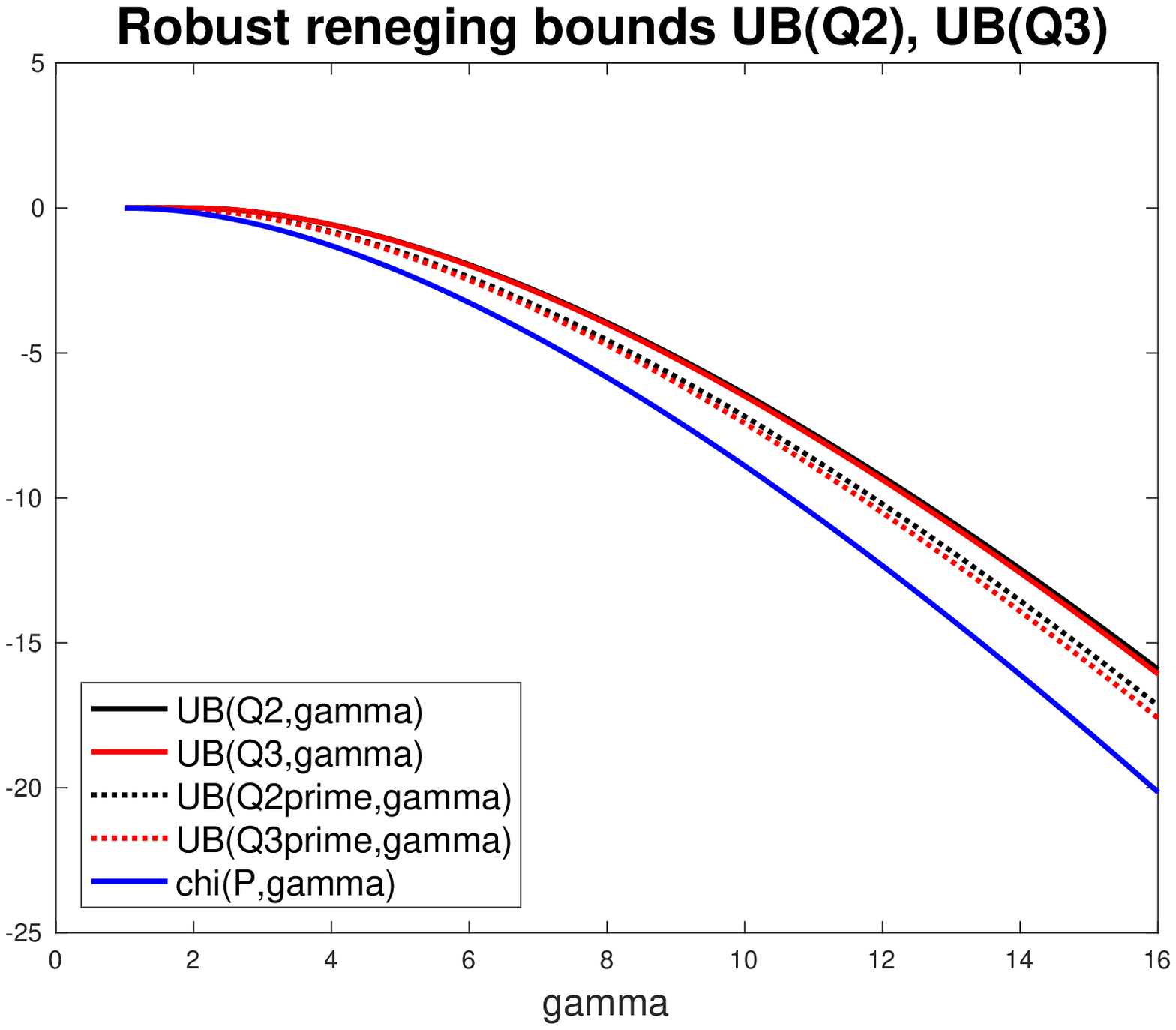}
\includegraphics[width=0.48\textwidth,height=0.45\textwidth]{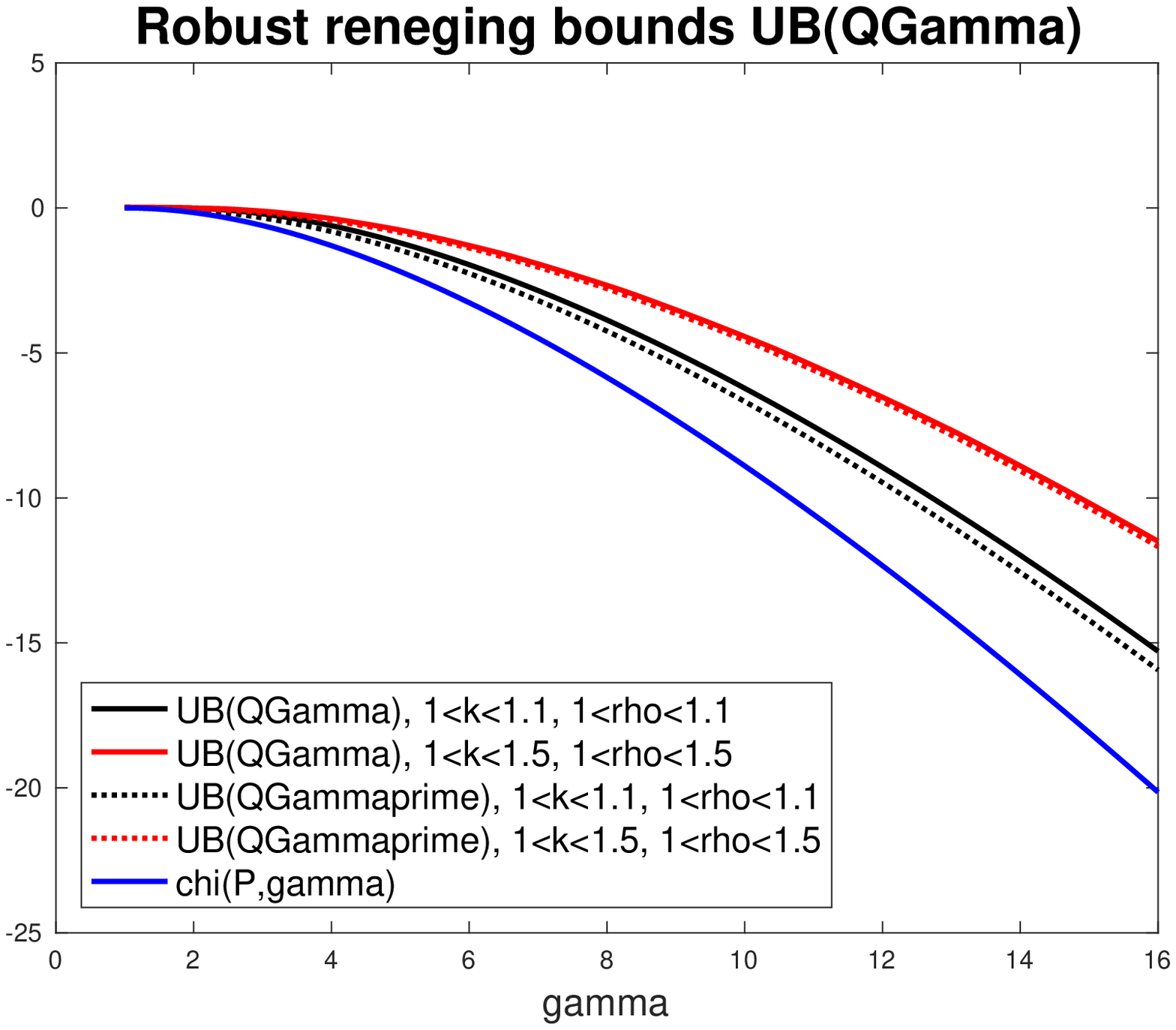}
}
\caption{Robust bounds for the reneging problem}
\label{fig3}
\end{figure}

\section{Concluding remarks}\label{sec6}

The techniques developed in this paper are not limited to queueing models.
The basic bound \eqref{091}
can be used in far broader dynamical system settings.
In the most general terms, its usefulness relies on the ability
to provide (1) a LD estimate
under some reference measure $P$ (the first term on the RHS of \eqref{091})
and (2) a computation of, or an effective upper bound on,
the RDR for a family of models of interest (second term on the RHS of \eqref{091}).
For example, if the dynamical systems
are driven by point processes (like in queueing applications),
the relevant RDR need not correspond to renewal versus Poisson like in this paper,
but between families of point processes relevant to the application.
Such RDR estimates need to be developed.

The main viewpoint presented in this paper was to consider
a reference model for which
computation is possible, and a family of models that need not be tractable.
A different perspective, initiated in \cite{dupkatpanrey}, is to use
these bounds for sensitivity analysis of rare event probabilities.
This paper introduces new gradient based
sensitivity indices that are meaningful at the large deviations scale, and
develops sensitivity bounds which
do not require a rare event sampler for each rare event.
This quality is closely related to the fact that in \eqref{091}
the difference in performance under two measures is bounded solely in terms
of the \renyi divergence, and does not depend on the rare event $A$.
%not tied to specific rare event simulation methods.
This method of \cite{dupkatpanrey} arguably has an advantage over more
traditional approaches of direct statistical estimation of rare event sensitivities.

%Another important viewpoint is to use the bounds in scenarios
%where an engineering system is designed assuming a {\it design} model $P$,
%for which the performance is not necessarily known in a mathematically rigorous fashion
%(say, simulations provide some indication about its performance),
%but bounds are required on the degree to which the performance
%of the {\it true} model, assumed to lie within a given family $\calQ$,
%deviates from that of the design model.
%(MAYBE PAUL COULD ADD REFERENCE TO SOME OF HIS PAPERS
%THAT TAKE THIS VIEWPOINT).
%In this case, it follows from the bound \eqref{091} that
%an estimate on the RDR of $\calQ$ w.r.t.\ $P$ will achieve this goal,
%thus providing guaranteed performance deviation bounds between true and design models.
%For such uses, there is no reason to take $P$ to be
%a Markov model.

Finally, the robust bounds for RS control developed in \S\ref{sec41}
are valid in far greater generality than for queueing applications, as we
have indeed emphasized in that section.
As long as the function $F$ in Theorem \ref{thm2} can be computed (or estimated) for a given design model $P$ and a specified family of models $\calQ$,  the
robust bounds established in this result are available.
%(I'M NOT SURE THIS SENTENCE IS NECESSARY,
%BUT I WANTED TO MENTION THE RS BUSINESS IN THE CONCLUDING REMARKS.
%ANYTHING ELSE WE COULD SAY ABOUT IT HERE?)

\appendix
\section{Appendix}
\manualnames{A}
\subsection{Proofs of results from \S \ref{sec31}}\label{sec-a1}

\noi{\bf Proof of Theorem \ref{th3}.}
(i)
For $T>0$, the Radon-Nikodym (RN) derivative of $Q^N_T$ w.r.t.\ $P^N_T$ is given by (see \cite[Theorem 2.31]{karr})
\begin{equation}\label{202}
\La_T=
e^{-\int_{[0,T]\times S}(\la(t,z)-\la_0) \mar(dz)\,dt+\int_{[0,T]\times S}\log\frac{\la(t,z)}{\la_0} N(dt\, dz)}.
\end{equation}
Raising this expression to the power $\al$ gives
\begin{align}\label{203}
\La_T^\al \notag
&=e^{-\int_{[0,T]\times S}(\la(t,z)^\al\la_0^{1-\al}-\la_0)\mar(dz) dt+\int_{[0,T]\times S}\log\frac{\la(t,z)^{\al}\la_0^{1-\al}}{\la_0}N(dt\,dz)}
e^{\int_{[0,T]\times S}(\la(t,z)^\al\la_0^{1-\al}-\al\la(t,z)+(\al-1)\la_0)\mar(dz)dt}
\\
&=M_Te^{\int_{[0,T]\times S}\la_0\al(\al-1)k_\al(\frac{\la(t,z)}{\la_0})\mar(dz)dt},
\end{align}
where the process
\[
M_t=e^{-\int_{[0,t]\times S}(\la(s,z)^\al\la_0^{1-\al}-\la_0)\mar(dz)ds+\int_{[0,t]\times S}\log\frac{\la(s,z)^{\al}\la_0^{1-\al}}{\la_0}N(ds\,dz)}, 0 \le t\le T
\]
is a $P$-(local) martingale. By the hypothesis on $\la(\cdot)$,
for all $Q\in\calQ_1$ and all $T$, $\int_{[0,T]\times S} k_\al(\la(t,z)/\la_0) \mar(dz)dt \le T(u+v(T))$.
Hence $R_\al(Q_T^N\|P_T^N)\le(u+v(T))T\la_0$.
Consequently, by the definition of the RDR, and since $v(T) \to 0$, $r_\al(\calQ_1\|P)$  is bounded above by $u\la_0$.
Equality follows on taking $\la(t,z)$ to be the constant $\la_1$ for which
$k_\al(\la_1/\la_0)=u$.

\skp

(ii) Fix $Q\in\calQ_2$. By the convexity of $k_\al$, the property \eqref{093} implies
that for all $t$, $k_\al(\la(t,z)/\la_0)\le k_\al(a)\vee k_\al(b)$, which using \eqref{203} yields
\begin{equation*}
%\label{307}
R_\al(Q_T\|P_T)\le[k_\al(a)\vee k_\al(b)]\la_0T.
\end{equation*}
This shows that $r_\al(\calQ_2\|P)\le (k_\al(a)\vee k_\al(b))\la_0$.  The equality in \eqref{201} now follows
by taking $\la(t,z)=a\la_0$ or $\la(t,z)=b\la_0$.

As for the claim regarding delayed renewal processes, it is well known (see e.g. \cite[Exercise 2.14]{karr})
that the RN derivative is given by
\begin{equation}\label{eq:eq822}
\La_T=e^{-\int_{[0,T]\times S}(h(V_t)\psi(z)-\la_0)\mar(dz)dt+\int_{[0,T]\times S}\log \frac{h(V_{t-})\psi(z)}{\la_0}N(dt\,dz)},
\end{equation}
where $V_t=t-\tau_{N_t(S)}$ is the backward recurrence time.
Hence the process $\la(t,z) \doteq h(V_{t-})\psi(z)$ is the  intensity process for the marked point process
$N$, satisfying the hypothesis $a\le\frac{\la(\cdot)}{\la_0}\le b$. The result now follows from the first part of (ii).

\skp

(iii)
Fix $Q\in\calQ_3$. Since $k_\alpha(x)$ is convex, we have
$$k_\alpha(x) \le \frac{x-a}{b-a} k_\alpha(b) + \frac{b-x}{b-a} k_\alpha(a).$$
Therefore
\begin{align*}
	\frac{1}{T} \int_{[0,T]\times S}k_\al\left(\frac{\la(t,z)}{\la_0}\right)\mar(dz)dt & \le \frac{1}{T} \int_{[0,T]\times S} \left( \frac{\frac{\la(t,z)}{\la_0}-a}{b-a} k_\alpha(b) + \frac{b-\frac{\la(t,z)}{\la_0}}{b-a} k_\alpha(a) \right) \mar(dz)dt \\
	& \le \frac{\frac{\la_0+v(T)}{\la_0}-a}{b-a} k_\alpha(b) + \frac{b-\frac{\la_0-v(T)}{\la_0}}{b-a} k_\alpha(a) \\
	& = pk_\alpha(a)+qk_\alpha(b) + \frac{v(T)(k_\alpha(a)+k_\alpha(b))}{\lambda_0(b-a)}
\end{align*}
by \eqref{096}.
It then follows from \eqref{203} and the $P$-(local) martingale property of $M$ that
$$\frac{1}{T}\frac{1}{\al(\al-1)}\log E_P[\La_T^\al]\le (pk_\alpha(a)+qk_\alpha(b))\lambda_0 + \frac{v(T)(k_\alpha(a)+k_\alpha(b))}{b-a}.$$
Taking the supremum over $Q \in \calQ_3$ and the limsup as $T \to \infty$, it follows that $r_\al(\calQ_3\|P)\le (pk_\al(a)+qk_\al(b))\la_0$.
To obtain the asserted equality, take, for each $T$,
deterministic $\la(\cdot)$ that takes the value $a\la_0$ (resp., $b\la_0$)
on $[0,pT)$ (resp., $[pT,T]$).

\skp

(iv)
Fix $1<\al<\al_0$.
Let $Q\in\calQ_4$ and let $\bar\la(\cdot)$ be the
normalized intensity process $\bar\la(\cdot)=\la(\cdot)/\la_0$.
Then for all $T>0$,
\begin{equation*}
%\label{190}
\frac{1}{T}\int_{[0,T]\times S}k_{\al_0}(\bar\la(t,z))\mar(dz) dt\le u+v(T),
\qquad
\Big|\frac{1}{T}\int_{[0,T]\times S}\bar\la(t,z)\mar(dz)dt-1\Big|\le\bar v(T)=\frac{v(T)}{\la_0}.
\end{equation*}
Denote by $\calG$ the collection of all (deterministic)
maps $f:\R_+\times S \to \R_+$ such that for all $T\in (0,\infty)$
\begin{equation}\label{calG}
    \frac{1}{T}\int_{[0,T]\times S}k_{\al_0}(f(t,z))\mar(dz) dt\le u+v(T),
\qquad
\Big|\frac{1}{T}\int_{[0,T]\times S} f(t,y)\mar(dy)dt-1\Big|\le\bar v(T)
\end{equation}
and let
\[
U_T=\sup\Big\{\int_{[0,T]\times S} k_\al(f(t,z))\mar(dz)dt: f \in \calG\Big\}.
\]
Since the normalized {\it stochastic} intensity $\bar\la(\cdot)$ is in $\calG$ a.s.,
$\int_{[0,T]\times S} k_\al(\bar\la(t,z))\mar(dz)dt\le U_T$, for every $T$.
Consequently, by \eqref{203}, $\La_T^\al\le M_Te^{\la_0\al(\al-1)U_T}$, for every $T$.
Since $M$ is a nonnegative local martingale,
\begin{equation}\label{092}
\frac{1}{T}\frac{1}{\al(\al-1)}\log E_P[\La_T^\al]\le\frac{\la_0U_T}{T}.
\end{equation}

We now compute $U_T$.
%
%Denote
%\[
%\tilde U=\sup\Big\{\int_{[0,T]\times S} k_\al(\bar\la(t,z))\mar(dz) dt: \bar\la \text{ is
%{\em deterministic} and satisfies \eqref{190}}\Big\}.
%\]
%Then arguing as in the proof of (iii) shows that \eqref{092} holds with
%$\tilde U$ in place of $U$, and it remains to compute $\tilde U$.
To this end, for $f\in\calG$, let the measure $\mu=\mu_T$
be the corresponding empirical measure on $[0,T]\times S$, namely,
\[
\mu(B)=\frac{1}{T}\int_{[0,T]\times S}1_{\{f(t,z)\in B\}}\mar(dz)dt,
\qquad B \in \calB([0,\infty)).
\]
Let the $p$th moment of $\mu$ be denoted by $m_p(\mu)=\int_{[0,\iy)}x^pd\mu(x)$.
Then by \eqref{calG}, $\lan k_{\al_0},\mu\ran \doteq \int k_{\al_0} d\mu \le u+v(T)$ and $|m_1(\mu)-1|\le \bar v(T)$.
The computation proceeds in two steps.
First we solve the problem of maximizing $m_\al(\mu)$ under the constraints
that $m_1(\mu)$ and $m_{\al_0}(\mu)$ are given.
Then we translate it into the problem of
maximizing $\int k_\al(x)d\mu(x)$ subject to the constraints \eqref{calG}.

Let $a,b,k,l$ be positive constants satisfying $a+b=\al$, $ka=1$, $lb=\al_0$ and $k^{-1}+l^{-1}=1$.
Using H\"older inequality,
\begin{align*}
m_\al(\mu)=\int x^\al d\mu(x) = \int x^{a}x^bd\mu(x)\le m_1(\mu)^{1/k}m_{\al_0}(\mu)^{1/l}.
\end{align*}
Solving for $a,b,k,l$ gives $k=\frac{\al_0-1}{\al_0-\al}$, $l=\frac{\al_0-1}{\al-1}$
(and $a=k^{-1}$, $b=\al_0l^{-1}$).
Moreover, the inequality is tight, specifically
\begin{equation}\label{191}
\mu(dx)=p\del_0(dx)+q\del_c(dx),
\end{equation}
satisfies it with equality,
with $1-p=q=m_1C^{-\frac{1}{\al_0-1}}$ and $c=C^{\frac{1}{\al_0-1}}$, where $C=\frac{m_{\al_0}}{m_1}$
(note: using the inequality $m_1^{\al_0}\le m_{\al_0}$ it is easy to check that $q\le 1$).

Next, recalling the notation $\bar\al=\al(\al-1)$ and $\bar\al_0=\al_0(\al_0-1)$,
\begin{align*}
\lan k_{\al},\mu\ran &= \frac{1}{\bar\al}(m_\al(\mu)-\al m_1(\mu)+\al-1)
\\ &\le\frac{1}{\bar\al}\Big(m_1(\mu)^{1/k}m_{\al_0}(\mu)^{1/l}-\al m_1(\mu)+\al-1\Big)
\\ &=\frac{1}{\bar\al}\Big[m_1(\mu)^{1/k}
\Big(\bar\al_0\lan k_{\al_0},\mu\ran+\al_0m_1(\mu)-\al_0+1\Big)^{1/l}
-\al m_1(\mu)+\al-1\Big].
\end{align*}
We now use the fact that $v\in\calV_0$, and the assumed bounds on $m_1(\mu)$
and $\lan k_{\al_0},\mu\ran$. We obtain
\begin{align*}
\lan k_{\al},\mu\ran
\le\frac{1}{\bar\al}\Big[\Big(\bar\al_0 u +1\Big)^{1/l}-1\Big] +\tilde v(T),
\end{align*}
for suitable $\tilde v\in\calV_0$, which depends on the parameter but not on $\mu$.
Combining with \eqref{092},
\[
\frac{1}{T}\frac{1}{\al(\al-1)}\log E[\La_T^\al]\le\frac{\la_0}{\bar\al}
\Big[\Big(\bar\al_0u+1\Big)^{\frac{\al-1}{\al_0-1}}-1\Big]+\la_0\tilde v(T).
\]
Taking supremum over $Q\in\calQ_4$ and the limsup as $T\to\iy$ gives
\[
r_\al(\calQ_4\|P)\le\frac{\la_0}{\bar\al}
\Big[\Big(\bar\al_0u+1\Big)^{\frac{\al-1}{\al_0-1}}-1\Big].
\]
Finally, equality is obtained by selecting, for each $T$, a deterministic
$\la(\cdot)$ that agrees with \eqref{191} in the sense that the empirical measure
$\mu$ corresponding to $\bar\la=\la/\la_0$ is given by \eqref{191}.
\qed

\subsection{Proofs of results from \S \ref{sec32}}\label{sec-a2}

Before presenting the proof of Theorem \ref{thm1}, we state and prove the following lemma.
Recall the notation $P^N_t$ and $Q^N_t$ from Section \ref{sec31}.
Let
$0= \tau_0< \tau_1 < \tau_2 < \cdots$ denote the occurrence times of the point process $N$
and let  for $i \in \N$, $\Del_i=\tau_i-\tau_{i-1}$.
Write $N_t$ for $N_t(S)$ for short when there is no ambiguity.

\begin{lemma}\label{lem:hzero}
	Assume that $\bar H\doteq\sup_{x\in \R_+} H(x)<\iy$.
	Also suppose that
		$c(\al) \doteq \int_S (\psi^{\al}(z)-1)\mar(dz)<\infty$.
	Then for every $\alpha > 1$,
	\begin{equation*}
		%\label{eq:thm1p}
	 	\frac{1}{t}R_{\alpha}(Q_t^N\|P_t^N) = \frac{1}{\al(\al-1)t}\log E_P[e^{\al\sum_{i=1}^{N_t}H(\Del_i)}]+ \frac{c(\al)}{\al(\al-1)} + o_t(1)
	\end{equation*}
	as $t\to \infty$.
\end{lemma}

\begin{proof}
	For fixed $t\ge 0$, let  $\eta_t \doteq \tau_{N_t+1} = \inf \{ s > t : N_s > N_t \}$. Then $\eta_t$ is a $\{\calF_s\}_{s\ge 0}$-stopping time.
Recall the notation $\La_t = \frac{dQ_t^N}{dP_t^N}$ and the expression \eqref{eq:eq822}.
	By the optional sampling theorem it follows that
	$E_P(\La_{\eta_t})=1$ for every $t\ge 0$ and for every $s\ge 0$ and $A \in \calF_{s \wedge \eta_t}$
	$$E_P(1_A \La_{\eta_t}) = E_P(1_A \La_{s \wedge  \eta_t }) = E_P(1_A \La_{s}) = E_Q(1_A).$$
	By a monotone class argument we now have that, with
	 $K_s \doteq N_{s \wedge \eta_t}(S)$ for $s \ge 0$ and $\Gmc \doteq \sigma \{ K_s : s \ge 0\}$,
% 	We claim that
	\begin{equation*}
		%\label{eq:claim_G}
		E_P [1_A \La_{\eta_t}] = E_Q [1_A], \qquad\forall A \in \Gmc.
	\end{equation*}
	Since  $\sigma\{N_s : 0 \le s \le t\}$ is contained in $\sigma\{ K_s : s \ge 0\}$, by the data processing inequality
	\cite[Theorem 1.24 and Corollary 1.29]{lievaj} and \cite[Sec II]{ervhar}
	we have
	\begin{align*}
		R_\al(Q_t^N\|P_t^N)&  = R_\alpha(Q \circ N_{[0,t]}^{-1} \| P \circ N_{[0,t]}^{-1}) \\
		&\le R_\alpha(Q \circ \{ K_s : s \ge 0\}^{-1} \| P \circ \{ K_s : s \ge 0\}^{-1}) \le \frac{\log E_P [\La^\alpha_{\eta_t}]}{\alpha(\alpha-1)}.
	\end{align*}
	Denote by $\{\xi_i\}$ the sequence of marks associated with the point process. 

	Using the expression of $\La_t$ in \eqref{eq:eq822} and the definition of $H$ in \eqref{eq:H}, we have
	\begin{align*}
		\La_{\eta_t}^\alpha & = \exp \Big\{ \al \sum_{i=1}^{N_t+1} \Big( \int_{\tau_{i-1}}^{\tau_i} (1-h(V_s))\,ds + \log h(\tau_i-\tau_{i-1}) + \log \psi(\xi_i) \Big) \Big\}\\
		& = \exp\Big\{\al \sum_{i=1}^{N_t+1} H(\Del_i)\Big\}\exp\Big\{\al \sum_{i=1}^{N_t+1} \log \psi(\xi_i)\Big\} , \; t\ge 0.
	\end{align*}
		Since marks are independent of jump instants and $H$
 is bounded from above, we have 
 $$
 \frac{1}{t} \log  E_P \La_{\eta_t}^\alpha = \frac{1}{t} \log E_P[e^{\al\sum_{i=1}^{N_t}H(\Del_i)}] + \frac{1}{t} \log E_P[e^{\al\sum_{i=1}^{N_t+1}\log \psi(\xi_i)}]
 +o_t(1).$$
 Also, by standard Laplace transform formulas (see e.g. \cite[Example 1.16]{karr})
 $$
 \frac{1}{t} \log E_P[e^{\al\sum_{i=1}^{N_t+1}\log \psi(\xi_i)}] = c(\al) + o_t(1).$$
 The result follows. \hfill \qed
\end{proof}

\skp

\noi{\bf Proof of Theorem \ref{thm1}.}
(a) From Lemma \ref{lem:hzero}
\begin{equation}
	\label{eq:thm1}
 	\frac{1}{t}R_{\alpha}(Q_t^N\|P_t^N)= \frac{1}{\al(\al-1)t}\log E_P[e^{\al\sum_{i=1}^{N_t}H(\Del_i)}] + \frac{c(\al)}{\al(\al-1)} + o_t(1).
\end{equation}

Bounding $H(\Del_i)$ by $\bar H$, we have $E_P[e^{\al\sum_{i=1}^{N_t}H(\Del_i)}]\le E_P[e^{\al\bar H{N_t}}]
=e^{t(e^{\al\bar H}-1)}$ and therefore
\begin{equation*}
	\limsup_t \frac{1}{t}R_{\alpha}(Q_t^N\|P_t^N)\le\frac{e^{\al \bar H}-1 + c(\al)}{\al(\al-1)}.
\end{equation*}
This gives the  bound \eqref{eq:rough_bd}.
%
% (b)
% Note that for a nonnegative RV $X$ and $p, q>1$ such that $p^{-1}+ q^{-1}=1$
% $$ E(X) \le (EX^q)^{1/q} \le \frac{(EX^q)^{p/q}}{p} + \frac{1}{q}.$$
% Applying this inequality to
% $X= e^{\al H(\Del_1)}$ we have
% $$\gamma(\al) \le \frac{1}{p}[\gamma(q\al)]^{p/q} + \frac{1}{q}.$$
% Thus
% $$\gamma(\al)-1 \le \frac{1}{p}[\gamma(q\al)]^{p/q} + \frac{1}{q} - 1 = \frac{[\gamma(q\al)]^{p/q}-1}{p}.$$
%
% This proves (b).
%
% **************************

For parts $(b)$, $(c)$ and $(d)$, we will need  a more careful analysis of $E_P[e^{\al\sum_{i=1}^{N_t}H(\Del_i)}]$, under different assumptions. 

Fix  $0<c_0<c_1$ and write
\begin{equation}\label{10}
E_P[e^{\al\sum_{i=1}^{N_t}H(\Del_i)}] =\sum_{k=0}^\iy E_P[1_{\{N_t=k\}}e^{\al\sum_{i=1}^kH(\Del_i)}] =Q_0(t)+Q_1(t)+Q_2(t),
\end{equation}
where
\begin{align*}
Q_0(t)
&=\sum_{k\le c_0t} E_P[1_{\{N_t=k\}}e^{\al\sum_{i=1}^kH(\Del_i)}],\\
Q_1(t)&=\sum_{k\ge c_1t}E_P[1_{\{N_t=k\}}e^{\al\sum_{i=1}^kH(\Del_i)}]\\
Q_2(t)&=\sum_{c_0t<k<c_1t}E_P[1_{\{N_t=k\}}e^{\al\sum_{i=1}^kH(\Del_i)}].
\end{align*}

Using the bound  $Q_0(t)\le e^{\al\bar H c_0t}$, we have
\begin{equation}\label{11}
\limsup_t\frac{1}{t}\log Q_0(t)\le\al\bar Hc_0.
\end{equation}
For bounding $Q_1(t)$, write
\begin{equation*}
Q_1(t)\le\sum_{k\ge c_1t}e^{\al\bar H k}e^{-t}\frac{t^k}{k!}\le c_{\rm st}e^{-t}\sum_{k\ge c_1t}e^{\al\bar H k}t^k e^k k^{-k-\frac{1}{2}},
\end{equation*}
where Stirling's approximation is used, and $c_{\rm st}$ is a universal constant.
Using $t/k\le 1/c_1$ for the summands in the above display, we have
\begin{align*}
Q_1(t)
&\le c_{\rm st}e^{-t}\sum_{k\ge c_1t}e^{\al\bar H k}e^k c_1^{-k}.
\end{align*}
For $c_1 \ge 2e^{\al\bar H+1}$, we have $c_1^{-1}e^{\al\bar H+1}<\frac{1}{2}$ and
$Q_1(t)\le c_{\rm st}$. Hence
\begin{equation}
  \label{9}
  \limsup_t\frac{1}{t}\log Q_1(t)\le 0, \quad \forall c_1 \ge 2e^{\al\bar H+1}.
\end{equation}

We now estimate $Q_2(t)$, using different approaches under different assumptions in parts $(b)$, $(c)$ and $(d)$.

(b) Note that
\begin{align*}
	\frac{1}{t} \log Q_2(t) & \le \frac{1}{t} \log \Big( ((c_1-c_0)t+1) \max_{k\in[c_0t,c_1t]}
	E_P[1_{\{N_t=k\}}e^{\al\sum_{i=1}^kH(\Del_i)}] \Big) \\
	& = o_t(1) + \frac{1}{t} \log\max_{k\in[c_0t,c_1t]}
	E_P[1_{\{N_t=k\}}e^{\al\sum_{i=1}^kH(\Del_i)}].
\end{align*}
Recall that $\gamma(s) \doteq E_P e^{sH(\Del_1)}$ for $s \in \Rmb$.
Let $p>0$ and $q>0$ be such that \ $1/p+1/q=1$. 
Then for each $k\in[c_0t,c_1t]$,
\begin{align*}
E_P[1_{\{N_t=k\}}e^{\al\sum_{i=1}^kH(\Del_i)}]
&\le P(N_t=k)^{1/p}\gamma(q\al)^{k/q}
\\ &=
(e^{-t}t^k/k!)^{1/p}\gamma(q\al)^{k/q}
\\ &\le
(e^{-t}t^k c_{\rm st}e^kk^{-k-\frac{1}{2}})^{1/p}\gamma(q\al)^{k/q}.
\end{align*}
Letting $\theta \doteq k/t$
\begin{align*}
  \frac{1}{t}\log E_P[1_{\{N_t=k\}}e^{\al\sum_{i=1}^kH(\Del_i)}]
&\le
-\frac{1}{p}+\frac{\theta}{p}\log t+\frac{\theta}{p}-\frac{\theta}{p}\log(\theta t)
+\frac{\theta}{q}\log\gamma(q\al)+o_t(1)
\\ &=
\frac{\theta-1}{p}-\frac{\theta}{p}\log\theta
+\frac{\theta}{q}\log\gamma(q\al)+o_t(1).
\end{align*}
We now maximize the sum of the first three terms on the last line  over $\theta$ (with $p$ and $q$  fixed).
The maximum is attained at $\theta=\gamma(q\al)^{p/q}$.
If we plug in this value of $\theta$ we obtain that the maximum is given by
\[
\frac{\gamma(q\al)^{p/q}-1}{p} = \hat\gamma(p,q,\al).
\]
As a result,
\[
\limsup_t\frac{1}{t}\log Q_2(t)\le\inf_{p,q>1:p^{-1}+q^{-1}=1}\hat\gamma(p,q,\al) = G_\alpha^{(1)}.
\]
Combining this with the bounds \eqref{eq:thm1}, \eqref{10}, \eqref{11} and \eqref{9} gives
\begin{align*}
r_\al^N(Q\|P)
&\le \left[\frac{\bar H c_0}{\al-1}
\vee 0
\vee\frac{G_\alpha^{(1)}}{\al(\al-1)}\right] + \frac{c(\al)}{\al(\al-1)}.
\end{align*}
Sending $c_0\to 0$ and $c_1\to\iy$ gives \eqref{eq:renew_1}.

(c)
Given $\eps>0$ let $c_0=\theta_0<\theta_1<\cdots<\theta_J=c_1$
be a finite partition of $[c_0,c_1]$ satisfying $\theta_j-\theta_{j-1} = \eps$ for all $j\le J$.
Then
\begin{equation}
	\label{eq:Q2_Q2j}
	\limsup_t\frac{1}{t}\log Q_2(t)\le\max_{1\le j\le J}\limsup_t\frac{1}{t}\log Q_2^j(t),
\end{equation}
where
\[
Q_2^j(t)=E_P[1_{\{\theta_{j-1}t\le N_t < \theta_jt\}}e^{\al\sum_{i=1}^{N_t}H(\Del_i)}].
\]
Fix $j\le J$. Denote $n=\lce\theta_{j-1}t\rce$.
Let $S_n=\sum_{i=1}^n\Del_i$ and $S^H_n=\sum_{i=1}^nH(\Del_i)$.
Use bar to denote the normalized sum, as in $\bar S_n=n^{-1}S_n$.
Since $\beta$ is finite in a neighborhood of the origin, by Cram\'{e}r's theorem, $(\Sbar_n,\Sbar_n^H)$ has LDP with a good rate function
(see e.g. \cite[Corollary 6.1.6]{demzei}) $\beta^*$.
Note that $\{\theta_{j-1}t \le N_t < \theta_jt\} \subset \{ S_{\lce\theta_{j-1}t\rce}\le t, S_{\lfl \theta_jt\rfl}\ge t \}$.
Then
\begin{equation}
	\label{eq:Q2j_1}
	Q_2^j(t) \le e^{\al\bar H\eps t} E_P[1_{\{S_{\lce\theta_{j-1}t\rce}\le t, S_{\lce \theta_jt\rce}\ge t\}}e^{\al\sum_{i=1}^nH(\Del_i)}] \le e^{\al\bar H\eps t}E_P[1_{\{S_n\le t\}}e^{\al S_n^H}].
\end{equation}
Let $g$ be the upper semicontinuous function defined as $g(x_1,x_2)=0$ for $0 \le x_1 \le \theta_{j-1}^{-1}$ and $g(x_1,x_2)=-\iy$ otherwise.
Then $1_{\{S_n\le t\}} \le e^{ng(\Sbar_n,\Sbar_n^H)}$.
Hence
\[
Q_2^j(t)\le e^{\al\bar H\eps t}E_P[e^{n(\al \bar S^H_n+g(\Sbar_n,\Sbar_n^H))}].
\]
Since $\Hbar<\infty$ the conditions of Varadhan's integral lemma (see e.g. \cite[Lemma 4.3.6]{demzei}) are valid and hence
\begin{equation}
	\label{eq:Qj2_1_bd}
	\limsup_t\frac{1}{t}\log Q^j_2(t)\le\al\bar H\eps
	+\theta_{j-1}\sup_{x\in\R^2:0 \le x_1\le\theta_{j-1}^{-1}}
	\Big[\al x_2-\beta^*(x)\Big] 
	\le \al\bar H\eps + G_\al^{(2)}(\theta_{j-1}).
\end{equation}

From this and \eqref{eq:Q2_Q2j} we have
\begin{equation*}
	\limsup_t\frac{1}{t}\log Q_2(t) \le\max_{1\le j\le J}\limsup_t\frac{1}{t}\log Q_2^j(t) \le \al\bar H\eps + \sup_{\theta \in [c_0,c_1]} G_\al^{(2)}(\theta).
\end{equation*}
Combining this with the bounds \eqref{eq:thm1}, \eqref{10}, \eqref{11} and \eqref{9} and sending $\eps\to 0$ gives
\begin{align*}
\limsup_t \frac{1}{t}R_{\alpha}(Q_t^N\|P_t^N)
&\le\left[\frac{\bar H c_0}{\al-1}
\vee 0
\vee\sup_{\theta\in[c_0,c_1]}\frac{G_\al^{(2)}(\theta)}{\al(\al-1)}\right] + \frac{c(\al)}{\al(\al-1)}.
\end{align*}
Sending $c_0\to 0$ and $c_1\to\iy$ gives \eqref{eq:renew_2}.

(d) Let  $m = \lfl \theta_j t \rfl$ and $S_m, S^H_m, \bar S_m, \bar S_m^H$ be defined in a similar manner as in part (c). Once more we use the fact that  $(\Sbar_m,\Sbar_m^H)$ has  a LDP with a rate function $\beta^*$.
Note that besides the bound \eqref{eq:Q2j_1}, we also have, with $p^{-1}+q^{-1}=1$, $p,q>1$,
\begin{equation}
	\label{eq:Q2j_2}
	Q_2^j(t) \le \left( E_P[1_{\{S_n\le t, S_m\ge t\}}e^{p\al\sum_{i=1}^mH(\Del_i)}] \right)^{1/p} \left( E_P[1_{\{\theta_{j-1}t\le N_t <\theta_jt\}}e^{-q\al\sum_{i=N_t+1}^mH(\Del_i)}] \right)^{1/q}.
\end{equation}
For the first term on the right hand side, we  apply Varadhan's integral lemma as in the proof of part $(c)$ and get
\begin{equation*}
%	\label{eq:Qj2_bd}
	\limsup_t\frac{1}{t}\log \left( E_P[1_{\{S_n\le t, S_m\ge t\}}e^{p\al\sum_{i=1}^mH(\Del_i)}] \right)^{1/p} \le \frac{\theta_j}{p}\sup_{x\in\R^2:x_1\ge\theta_j^{-1}}
	\Big[p\al x_2-\beta^*(x)\Big].
\end{equation*}
For the second term on the right hand side of \eqref{eq:Q2j_2}, we have
\begin{align*}
	& \limsup_t\frac{1}{t}\log \left( E_P[1_{\{\theta_{j-1}t\le N_t<\theta_jt\}}e^{-q\al\sum_{i=N_t+1}^mH(\Del_i)}] \right)^{1/q} \\
	& \le \limsup_t\frac{1}{qt}\log \left( [(\theta_j-\theta_{j-1})t+1] \max_{\theta_{j-1}t\le k <\theta_jt} [\gamma(-q\alpha)]^{m-k} \right) \\
	& \le \limsup_t\frac{1}{qt}\log \left( [\eps t+1] ([\gamma(-q\alpha)]^{\eps t} \vee 1) \right) \\
	& = \frac{\log^+\gamma(-q\alpha)}{q}\eps.
\end{align*}
Combining these two bounds with \eqref{eq:Q2j_2} gives
\begin{equation*}
	\limsup_t\frac{1}{t}\log Q_2^j(t) \le \theta_j \sup_{x\in\R^2:x_1\ge\theta_j^{-1}} \Big[\al x_2-\frac{1}{p} \beta^*(x)\Big] + \frac{\log^+\gamma(-q\alpha)}{q}\eps.
\end{equation*}
Combining this with \eqref{eq:Qj2_1_bd}, we have
\begin{align*}
	\limsup_t\frac{1}{t}\log Q^j_2(t)
	& \le \left( \al\bar H\eps +\theta_{j-1}\sup_{x\in\R^2:0 \le x_1\le\theta_{j-1}^{-1}} \Big[\al x_2-\beta^*(x)\Big] \right)  \\
	& \qquad \wedge \left( \theta_j \sup_{x\in\R^2:x_1\ge\theta_j^{-1}} \Big[\al x_2-\frac{1}{p} \beta^*(x)\Big] + \frac{\log^+\gamma(-q\alpha)}{q}\eps\right) 
	\end{align*}
Since $p>1$, we have
\begin{align}
	\limsup_t\frac{1}{t}\log Q^j_2(t)	
	& \le \al\bar H\eps + \frac{\log^+\gamma(-q\alpha)}{q}\eps + \theta_j \left( \sup_{x\in\R^2:0 \le x_1\le\theta_{j-1}^{-1}} \Big[\al x_2-\frac{1}{p}\beta^*(x)\Big] \right. \notag \\
	& \qquad \wedge \left. \sup_{x\in\R^2:x_1\ge\theta_j^{-1}} \Big[\al x_2-\frac{1}{p} \beta^*(x)\Big] \right) \notag \\
	& \le \al\bar H\eps + \frac{\log^+\gamma(-q\alpha)}{q}\eps + \theta_j \sup_{\theta\in[\theta_{j-1},\theta_j]} \left( \sup_{x\in\R^2:0 \le x_1\le\theta^{-1}} \Big[\al x_2-\frac{1}{p}\beta^*(x)\Big] \right. \notag \\
	& \qquad \left. \wedge \sup_{x\in\R^2:x_1\ge\theta^{-1}} \Big[\al x_2-\frac{1}{p} \beta^*(x)\Big] \right), \label{eq:Q2j_1_complicate}
\end{align}
where the second inequality  uses the fact that for a function $r$ defined on $\R$ and constants $a<b$,
one has
\[
\sup_{y\in(-\iy,b]}r(y)\w\sup_{y\in[a,\iy)}r(y)\le\sup_{z\in[a,b]}
\left(\sup_{y\in(-\iy,z]}r(y)\w\sup_{y\in[z,\iy)}r(y)\right).
\]
Since $x \mapsto x_2-\frac{1}{p} \beta^*(x)$ is a concave function on $\Rmb^2$, the last term in \eqref{eq:Q2j_1_complicate} equals
\begin{align*}
	\theta_j \sup_{\theta\in[\theta_{j-1},\theta_j]} \sup_{x_2\in\R} \Big[\al x_2-\frac{1}{p}\beta^*(\theta^{-1},x_2)\Big] 
	& \le \frac{\theta_j}{\theta_{j-1}} \sup_{\theta\in[\theta_{j-1},\theta_j]} \theta \sup_{x_2\in\R} \Big[\al x_2-\frac{1}{p}\beta^*(\theta^{-1},x_2)\Big] \\
	& \le (1+\frac{\eps}{c_0})\sup_{\theta\in[\theta_{j-1},\theta_j]} \theta \sup_{x_2\in\R} \Big[\al x_2-\frac{1}{p}\beta^*(\theta^{-1},x_2)\Big].
\end{align*}
From this, \eqref{eq:Q2_Q2j} and \eqref{eq:Q2j_1_complicate}, letting
\begin{equation*}
	%\label{eq:G_3_p}
	G_\al^{(3)}(p,\theta) \doteq \theta \sup_{x_2\in\R} \Big[\al x_2-\frac{1}{p}\beta^*(\theta^{-1},x_2)\Big],
\end{equation*}
we have
\begin{equation*}
	\limsup_t\frac{1}{t}\log Q_2(t) \le\max_{1\le j\le J}\limsup_t\frac{1}{t}\log Q_2^j(t) \le \al\bar H\eps + \frac{\log^+\gamma(-q\alpha)}{q}\eps + (1+\frac{\eps}{c_0}) \sup_{\theta \in [c_0,c_1]} G_\al^{(3)}(p,\theta).
\end{equation*}
Combining this with the bounds \eqref{eq:thm1}, \eqref{10}, \eqref{11} and \eqref{9} and sending $\eps\to 0$ gives
\begin{equation*}
	%\label{eq:G3_total}
	\limsup_t\frac{1}{t} R_{\alpha}(Q_t^N\|P_t^N)
	\le\left[\frac{\bar H c_0}{\al-1}
	\vee 0
	\vee\sup_{\theta\in[c_0,c_1]}\frac{G_\al^{(3)}(p,\theta)}{\al(\al-1)}\right] + \frac{c(\al)}{\al(\al-1)}.
\end{equation*}
Now we claim that
\begin{equation}
	\label{eq:G_claim}
	\lim_{p \to 1} \sup_{\theta\in[c_0,c_1]}G_\al^{(3)}(p,\theta) = \sup_{\theta\in[c_0,c_1]}G_\al^{(3)}(\theta).
\end{equation}
Once this claim is verified, sending $p \to 1$, $c_0\to 0$ and $c_1\to\iy$ gives \eqref{eq:renew_3}.

It remains to prove the claim \eqref{eq:G_claim}.
First note that for any $\theta_0 \in [c_0,c_1]$ and $x_2 \in \Rmb$,
\begin{equation*}
	\liminf_{p \to 1} \sup_{\theta\in[c_0,c_1]}G_\al^{(3)}(p,\theta) \ge \liminf_{p \to 1} \theta_0 \Big[\al x_2-\frac{1}{p}\beta^*(\theta^{-1}_0,x_2)\Big] = \theta_0 \Big[\al x_2-\beta^*(\theta^{-1}_0,x_2)\Big].
\end{equation*}
Taking supremum over $x_2 \in \Rmb$ and $\theta_0 \in [c_0,c_1]$ gives
\begin{equation}
	\label{eq:G_claim_liminf}
	\liminf_{p \to 1} \sup_{\theta\in[c_0,c_1]}G_\al^{(3)}(p,\theta) \ge \sup_{\theta \in [c_0,c_1]} G_\alpha^{(3)}(\theta).
\end{equation}
Since $\beta^*$ is a good rate function, we can find $\kappa_0 \in (-\infty, 0)$ such that, for all $p \in [1/2,2]$,
%$\sup_{\theta \in [c_0,c_1]} G_\alpha^{(3)}(\theta) > -\infty$, 
%we can assume without loss of generality that 
$\sup_{\theta\in[c_0,c_1]}G_\al^{(3)}(p,\theta) \ge \kappa_0$.

Next we show $\limsup_{p \to 1} \sup_{\theta\in[c_0,c_1]}G_\al^{(3)}(p,\theta) \le \sup_{\theta \in [c_0,c_1]} G_\alpha^{(3)}(\theta)$.
%Assume without loss of generality that $G_{\alpha,\eps,c_0,c_1} \ge 0$.
For $p \in [1/2,2]$, $p\neq 1$, let $\theta_p \in [c_0,c_1]$ and $x_{2,p} \in \Rmb$ be such that
\begin{equation}
    \label{eq:kappa0}
	\kappa_0 \le \sup_{\theta\in[c_0,c_1]}G_\al^{(3)}(p,\theta) \le \theta_p \Big[\al x_{2,p}- \frac{1}{p}\beta^*(\theta^{-1}_p,x_{2,p}) \Big] + |p-1|.
\end{equation}
From $\beta^*\ge 0$ we have $x_{2,p} \ge \frac{\kappa_0-1}{c_0\alpha}$.
Since $\beta^*(x_1,x_2) \ge \lambda_2 x_2 - \beta(0,\lambda_2)$ for each $\lambda_2 > 0$, and $\beta(0,\la_2)<\infty$ for all $\la_2>0$, we have
 $\lim_{x_2 \to \infty} \inf_{x_1 \in \Rmb} \frac{\beta^*(x_1,x_2)}{x_2} = \infty$. This shows that $x_{2,p}$ is bounded from above, since
 if $x_{2,p}\to \infty$ as $p\to 1$ then from \eqref{eq:kappa0} we must have
 $$\limsup_{p\to 1} \frac{\beta^*(\theta^{-1}_p,x_{2,p}) }{px_{2,p}} \le \theta_p\al$$
 which is a contradiction. 
 Hence $x_{2,p} \le \kappa_1$ for some $\kappa_1 < \infty$ and
therefore the sequence $\{(\theta_p,x_2^p)\}$ is bounded.
Assume without loss of generality that $(\theta_p,x_2^p) \to (\thetabar,\xbar_2) \in [c_0,c_1]\times\Rmb$ along the whole subsequence.
Then
\begin{align*}
	\limsup_{p \to 1} \sup_{\theta\in[c_0,c_1]}G_\al^{(3)}(p,\theta) & \le \limsup_{p \to 1} \left( \theta_p \Big[\al x_{2,p}- \frac{1}{p}\beta^*(\theta^{-1}_p,x_{2,p}) \Big] + |p-1| \right) \\
	& = \thetabar \Big[\al \xbar_2- \liminf_{p \to 1}\beta^*(\theta^{-1}_p,x_{2,p}) \Big] \\
	& \le \thetabar \Big[\al \xbar_2- \beta^*(\thetabar^{-1},\xbar_2) \Big] \\
	& \le \sup_{\theta \in [c_0,c_1]} G_\alpha^{(3)}(\theta),
\end{align*}
where the second inequality follows from the lower semicontinuity of $\beta^*$.
Combining this with \eqref{eq:G_claim_liminf} gives the claim \eqref{eq:G_claim}.
This completes the proof.
\qed
\subsection{Proofs of results from \S \ref{sec:exam}}\label{sec-a3}
In this section we provide details of some of the calculations that were omitted from Section \ref{sec:exam}.\\

\noi{\bf Proofs of Statements in Example \ref{exa:exp}.}
We first consider the case $\rho>1$ and show that the right sides of \eqref{eq:renew_2} and \eqref{eq:renew_3} are the same and the inequalities in both cases can be replaced by equalities.

Note that $h(x) = \rho > 0$, $H(x) = -(\rho-1)x + \log \rho \le \log \rho < \infty$ and $\beta$ is finite in a neighborhood of the origin, namely all assumptions for \eqref{eq:renew_2} hold.
	Then, as follows from \eqref{202}, \eqref{203} in the appendix,
	\begin{align*}
	 	\frac{1}{t}R_{\alpha}(Q_t^N\|P_t^N) & = \frac{1}{\al(\al-1)t}\log E_P[\La_t^\al] = \frac{1}{\al(\al-1)t}\log E_P[e^{\alpha (1-\rho) t + \alpha N_t \log \rho}] \\
	 	& = \frac{1}{\al(\al-1)} \left[ \rho^\alpha - 1 - \alpha(\rho-1) \right]. 
	 	%\label{eq:Renyi_exp}
	\end{align*}
	Thus from Theorem \ref{thm1}(c) $ \rho^\alpha - 1 - \alpha(\rho-1)  \le \sup_{\theta \in (0,\infty)} G_\alpha^{(2)}(\theta)$.
	Now we show the reverse inequality.
	
	If $\theta \ge 1+\alpha(\rho-1)$, taking $\lambda_1 = 1+\alpha(\rho-1)-\theta \le 0$ in \eqref{eq:rmk_1} gives
	\begin{align*}
		G_\alpha^{(2)}(\theta) & \le -\lambda_1 + \theta \beta(\lambda_1,\alpha) = -1-\alpha(\rho-1)+\theta + \theta \log E_P[e^{(1-\theta) \Del + \alpha \log \rho}] \\
		& = -1-\alpha(\rho-1)+\theta + \theta [\alpha \log \rho - \log \theta] \le \rho^\alpha -1-\alpha(\rho-1),
	\end{align*}
	where the last inequality becomes equality when $\theta = \rho^\alpha$.
	If $0 < \theta < 1+\alpha(\rho-1)$, taking $\lambda_1 = 0$ in \eqref{eq:rmk_1} gives
	\begin{equation*}
		G_\alpha^{(2)}(\theta) \le \theta \beta(0,\alpha) = \theta \log E_P[e^{-\alpha(\rho-1) \Del + \alpha \log \rho}] = \theta \log \frac{\rho^\alpha}{1+\alpha(\rho-1)} \le [1+\alpha(\rho-1)] \log \frac{\rho^\alpha}{1+\alpha(\rho-1)},
	\end{equation*}	
	where the last inequality follows from the fact that $\rho^\alpha = (1+(\rho-1))^\alpha \ge 1+\alpha(\rho-1)$.
	Using $\ell(x) \doteq x\log x - x +1 \ge 0$, the term on the right side of the last display can be written as
	\begin{align*}
		- \rho^\alpha \ell(\frac{1+\alpha(\rho-1)}{\rho^\alpha}) - [1+\alpha(\rho-1)] + \rho^\alpha \le \rho^\alpha -1-\alpha(\rho-1).
	\end{align*}
	Therefore $\sup_{\theta \in (0,\infty)} G_\alpha^{(2)}(\theta) \le \rho^\alpha -1-\alpha(\rho-1)$.
	Thus we have shown that the  inequality  in \eqref{eq:renew_2} is in fact an equality.
	From this and the observation that $G_\alpha^{(2)}(\theta) \ge G_\alpha^{(3)}(\theta)$ we see that in \eqref{eq:renew_3} also the inequality can be replaced  with an equality.
	
Consider now the case $\rho \in (0,1]$.	 We show that once more the right sides of \eqref{eq:renew_2} and \eqref{eq:renew_3} are the same and the inequalities in both cases can be replaced by equalities.
	The proof of $\sup_{\theta \in (0,\infty)} G_\alpha^{(2)}(\theta) \le \rho^\alpha -1-\alpha(\rho-1)$ for \eqref{eq:renew_2} is exactly as before.
% 	To see this, note that the above arguments for \eqref{eq:renew_2} still hold when $\rho \in (0,1]$, that is	
% 	\begin{equation*}
% 		LHS\eqref{eq:renew_2} = \frac{\rho^\alpha - 1 - \alpha(\rho-1)}{\alpha(\alpha-1)} = RHS\eqref{eq:renew_2}.
% 	\end{equation*}
	For \eqref{eq:renew_3}, observe first  that
	\begin{align*}
		\beta(\lambda_1,\lambda_2) & = \lambda_2 \log \rho - \log \left[ 1-\lambda_1+\lambda_2(\rho-1) \right], \quad 1-\lambda_1+\lambda_2(\rho-1) > 0, \\
		\beta^*(x_1,\log \rho - (\rho-1)x_1) & = x_1-1-\log x_1, \quad x_1 > 0.
	\end{align*}
	Therefore
	\begin{align*}
		\sup_{\theta \in (0,\infty)} G_\alpha^{(3)}(\theta) & \ge G_\alpha^{(3)}(\rho^\alpha) = \rho^\alpha \sup_{x_2 \in \Rmb} \left[ \alpha x_2 - \beta^*(\frac{1}{\rho^{\alpha}},x_2) \right] \\
		& \ge \rho^\alpha \left[ \alpha (\log \rho - \frac{\rho-1}{\rho^\alpha}) - \beta^*(\frac{1}{\rho^{\alpha}},\log \rho - \frac{\rho-1}{\rho^\alpha}) \right] \\
		& = \rho^\alpha \left[ \alpha (\log \rho - \frac{\rho-1}{\rho^\alpha}) - (\frac{1}{\rho^{\alpha}}-1+\log \rho^{\alpha}) \right] = -\alpha(\rho-1)-1+\rho^\alpha,
	\end{align*}
	and hence
	\begin{equation*}
		LHS\eqref{eq:renew_3} = \frac{\rho^\alpha - 1 - \alpha(\rho-1)}{\alpha(\alpha-1)} = RHS\eqref{eq:renew_3} = RHS\eqref{eq:renew_2}.
	\end{equation*}
\qed	
\\ \ \\

\noi{\bf Proof of \eqref{eq:bdforgam} in Example \ref{ex:32}.}
	We will use Theorem \ref{thm1}(c)
	and establish \eqref{eq:bdforgam} by estimating
$\sup_{\theta \in (0,\infty)} G_\alpha^{(2)}(\theta)$.

	If $\theta \ge \frac{1+\alpha(\rho-1)}{1+\alpha(k-1)}$, taking $\lambda_1 = 1+\alpha(\rho-1)-\theta[1+\alpha(k-1)] \le 0$ in \eqref{eq:rmk_1} gives
	\begin{align*}
		G_\alpha^{(2)}(\theta) & \le -\lambda_1 + \theta \beta(\lambda_1,\alpha) \\
		& = -1-\alpha(\rho-1) + \theta[1+\alpha(k-1)] + \theta \left\{ \alpha k \log \rho + \log \left( \frac{\Gam(1+\alpha(k-1))}{(\Gam(k))^\alpha} \right) \right. \\
		& \qquad \left. - (1+\alpha(k-1)) \log [ \theta (1+\alpha(k-1)) ] \right\} \\
		& \le \left( \frac{\Gam(1+\alpha(k-1))}{(\Gam(k))^\alpha} \rho^{\alpha k} \right)^{\frac{1}{1+\alpha(k-1)}} -\alpha(\rho-1) - 1,
%		\rho^{\frac{\alpha k}{1+\alpha(k-1)}}
	\end{align*}
	where the last inequality is attained for $\theta = \theta^*$ that satisfies
	\begin{equation}
		\label{eq:eg_temp}
		\alpha k \log \rho + \log \left( \frac{\Gam(1+\alpha(k-1))}{(\Gam(k))^\alpha} \right) - (1+\alpha(k-1)) \log [ \theta^* (1+\alpha(k-1)) ] = 0.
	\end{equation}
	One can check that  $\theta^*$  is indeed greater than or equal to $\frac{1+\alpha(\rho-1)}{1+\alpha(k-1)}$.
To see this, note that the left hand side in \eqref{eq:eg_temp} is decreasing in $\theta$. So it suffices to check
	$$\alpha k \log \rho + \log \left( \frac{\Gam(1+\alpha(k-1))}{(\Gam(k))^\alpha} \right) - (1+\alpha(k-1)) \log [ (1+\alpha(\rho-1)) ] \ge 0.$$
	This is equivalent to checking
	$\beta(0,\alpha) \ge 0.$
	But this is immediate since
	$$\beta(0,\alpha) = \log \int (g(x)e^x)^\alpha e^{-x} \, dx \ge \log \left( \int g(x)e^x e^{-x} \, dx \right)^\alpha=0$$
	by Holder's inequality.

	If $0 < \theta < \frac{1+\alpha(\rho-1)}{1+\alpha(k-1)}$, taking $\lambda_1 = 0$ in \eqref{eq:rmk_1} gives
	\begin{equation}
		\label{eq:eg_temp2}
		G_\alpha^{(2)}(\theta) \le \theta \beta(0,\alpha) = \theta \left[ \alpha k \log \rho + \log \left( \frac{\Gam(1+\alpha(k-1))}{(\Gam(k))^\alpha} \right) - (1+\alpha(k-1)) \log (1+\alpha(\rho-1))  \right].
	\end{equation}	
	Since the left hand side in \eqref{eq:eg_temp} is decreasing in $\theta$, the expression obtained by replacing
	$\theta$ by $\frac{1+\alpha(\rho-1)}{1+\alpha(k-1)}$ in this term
	 is nonnegative, which shows that the term on the right side of \eqref{eq:eg_temp2} is nonnegative.
	Therefore
	\begin{align*}
		G_\alpha^{(2)}(\theta) & \le \frac{1+\alpha(\rho-1)}{1+\alpha(k-1)} \left[ \alpha k \log \rho + \log \left( \frac{\Gam(1+\alpha(k-1))}{(\Gam(k))^\alpha} \right) - (1+\alpha(k-1)) \log (1+\alpha(\rho-1))  \right] \\
		& = -[1+\alpha(\rho-1)] \log \left\{ [1+\alpha(\rho-1)] \Big/ \left( \frac{\Gam(1+\alpha(k-1))}{(\Gam(k))^\alpha} \rho^{\alpha k} \right)^{\frac{1}{1+\alpha(k-1)}} \right\} \\
		& = - \left( \frac{\Gam(1+\alpha(k-1))}{(\Gam(k))^\alpha} \rho^{\alpha k} \right)^{\frac{1}{1+\alpha(k-1)}} \ell\left( [1+\alpha(\rho-1)] \Big/ \left( \frac{\Gam(1+\alpha(k-1))}{(\Gam(k))^\alpha} \rho^{\alpha k} \right)^{\frac{1}{1+\alpha(k-1)}} \right) \\
		& \qquad - [1+\alpha(\rho-1)] + \left( \frac{\Gam(1+\alpha(k-1))}{(\Gam(k))^\alpha} \rho^{\alpha k} \right)^{\frac{1}{1+\alpha(k-1)}} \\
		& \le \left( \frac{\Gam(1+\alpha(k-1))}{(\Gam(k))^\alpha} \rho^{\alpha k} \right)^{\frac{1}{1+\alpha(k-1)}} - \alpha(\rho-1) - 1,
	\end{align*}
	where the third line uses the equality 
	\begin{equation}\label{eq:eq353}
		-a \log (a/b) = - b \ell(a/b) -a+b. \end{equation}
	Combining the above estimates with Theorem \ref{thm1}(c) we now have that when $\pi=$ Gamma$(k,\rho)$ with $k \ge 1$ and $\rho>1$
	$$ r_\al^N(Q\|P) \le \sup_{\theta \in (0,\infty)} \frac{1}{\alpha(\alpha-1)} G_\alpha^{(2)}(\theta) \le \frac{1}{\alpha(\alpha-1)} \left[ \left( \frac{\Gam(1+\alpha(k-1))}{(\Gam(k))^\alpha} \rho^{\alpha k} \right)^{\frac{1}{1+\alpha(k-1)}} -\alpha(\rho-1) - 1 \right].$$
\qed
\\ \ \\

\noi{\bf Proof of \eqref{eq:bdfphas} in Example \ref{lem:convert_to_exp}.}
Consider first $\theta \ge 1+\alpha(\sigma-1)$. In this case, taking $\lambda_1 = 1+\alpha(\sigma-1)-\theta \le 0$ in \eqref{eq:rmk_1}, we have
	\begin{equation*}
		G_\alpha^{(2)}(\theta) \le -\lambda_1 + \theta \beta(\lambda_1,\alpha) \le -1-\alpha(\sigma-1)+\theta + \theta [\alpha \log C - \log \theta] \le C^\alpha -1-\alpha(\sigma-1),
	\end{equation*}
	where the last inequality is attained when $\theta = C^\alpha$.
	Note that $C^{\alpha}$ is indeed in the range $[1+\alpha(\sigma-1), \infty)$. To see this note that
	\begin{equation*}
		1 = \int g(x) \, dx \le \int Ce^{-\sigma x} \, dx = \frac{C}{\sigma}
	\end{equation*}
	which shows rhat
	\begin{equation}
		\label{eq:convert_to_exp_temp}
		C^\alpha \ge \sigma^\alpha = (1+\sigma-1)^\alpha \ge 1 + \alpha(\sigma-1).
	\end{equation}
	
	Now consider the case $0 < \theta < 1+\alpha(\sigma-1)$. In this case, taking $\lambda_1 = 0$ in \eqref{eq:rmk_1}, we have
	\begin{equation*}
		G_\alpha^{(2)}(\theta) \le \theta \beta(0,\alpha) \le \theta \log \frac{C^\alpha}{1+\alpha(\sigma-1)} \le [1+\alpha(\sigma-1)] \log \frac{C^\alpha}{1+\alpha(\sigma-1)},
	\end{equation*}	
	where the last inequality follows from \eqref{eq:convert_to_exp_temp}.
	Recalling $\ell(x) = x\log x - x +1 \ge 0$ and using the equality in \eqref{eq:eq353} once more, from the last display we have
	for all $0 < \theta < 1+\alpha(\sigma-1)$
	\begin{equation*}
		G_\alpha^{(2)}(\theta) \le - C^\alpha \ell(\frac{1+\alpha(\sigma-1)}{C^\alpha}) - [1+\alpha(\sigma-1)] + C^\alpha \le C^\alpha -1-\alpha(\sigma-1).
	\end{equation*}
	Combining the above estimates with Theorem \ref{thm1}(c) we have the bound \eqref{eq:bdfphas} on RDR for this class of models.
\qed
\skp

\noi{\bf Acknowledgment.}
Research of RA supported in part by the ISF (grant 1184/16).
%Research of AB supported in part by the NSF (DMS--1814894, DMS--1853968).
Research of AB supported in part by the NSF (DMS-1305120, DMS-1814894, DMS-1853968).
Research of PD supported in part by the NSF (DMS-1904992) and 
AFOSR (FA-9550-18-1-0214).
Research of RW supported in part by the DARPA (W911NF-15-2-0122).

\setstretch{0.0}

%\bibliography{main}

%%\bibliographystyle{plain}
%%\bibliographystyle{annotate}
%%\bibliographystyle{apalike}
\bibliographystyle{is-abbrv}

\end{document}